\documentclass[11pt]{amsart}
\usepackage[active]{srcltx}
\usepackage{verbatim}
\usepackage{epsfig,graphicx,color,mathrsfs}
\usepackage{graphicx}
\usepackage{amsmath,amssymb,amsthm,amsfonts}
\usepackage{amssymb}
\usepackage[english]{babel}
\usepackage{epsfig,graphicx,color,mathrsfs}
\usepackage{esvect}
\usepackage{esint}

\usepackage[left=2.7cm,right=2.7cm,top=3cm,bottom=3cm]{geometry}

\numberwithin{equation}{section}
\newtheorem{theorem}{Theorem}[section]
\newtheorem{proposition}[theorem]{Proposition}
\newtheorem{lemma}[theorem]{Lemma}
\newtheorem{corollary}[theorem]{Corollary}
\newtheorem{definition}[theorem]{Definition}
\newtheorem{remark}[theorem]{Remark}
\newtheorem{example}[theorem]{Example}

\theoremstyle{definition}

\newcommand{\dyle}{\displaystyle}
\renewcommand{\dfrac}{\displaystyle\frac}
\newcommand{\brm}{\begin{remark}\rm}
\newcommand{\erm}{\end{remark}}
\newcommand{\brms}{\begin{remark}\rm}
\newcommand{\erms}{\end{remark}}
\newcommand{\bte}{\begin{theorem}}
\newcommand{\ete}{\end{theorem}}
\newcommand{\bpr}{\begin{proposition}}
\newcommand{\epr}{\end{proposition}}
\newcommand{\ble}{\begin{lemma}}
\newcommand{\ele}{\end{lemma}}
\newcommand{\beq}{\begin{equation}}
\newcommand{\eeq}{\end{equation}}
\newcommand{\bdm}{\begin{displaymath}}
\newcommand{\edm}{\end{displaymath}}
\numberwithin{equation}{section}

\newcommand{\bos}{\begin{remark}\rm}
\newcommand{\eos}{\end{remark}}

\newcommand{\ben}{\begin{enumerate}}
\newcommand{\een}{\end{enumerate}}

\renewcommand{\l }{\lambda }

\newcommand{\var }{\varphi }

\newcommand{\be}{\begin{equation}}
\newcommand{\ee}{\end{equation}}

\title[Neumann conditions for the higher order $s$-fractional Laplacian]{Neumann conditions for the higher order $s$-fractional Laplacian  $(-\Delta)^su$ with $s>1$}

\author[B. Barrios]{ Bego\~na Barrios}

\author[L.\ Montoro]{Luigi Montoro}

\author[I.\ Peral]{Ireneo Peral}

\author[F.\ Soria]{Fernando Soria}

\address{B. Barrios, Departamento de An\'{a}lisis Matem\'{a}tico, Universidad de La Laguna, C/. Astrofisico Francisco Sanchez s/n, 38200 Tenerife, Spain}
\email{bbarrios@ull.es}
\address{L. Montoro, Dipartimento di Matematica e Informatica, UNICAL, Ponte Pietro  Bucci 31B,
87036 Arcavacata di Rende, Cosenza, Italy}
\email{montoro@mat.unical.it}
\address{I. Peral,  Departamento de Matem\'{a}ticas, Universidad Aut\'{o}noma de Madrid,
28049 Madrid, Spain.}
\email{ireneo.peral@uam.es}
\address{F. Soria,  Departamento de Matem\'{a}ticas, Universidad Aut\'{o}noma de Madrid,
28049 Madrid, Spain.}
\email{fernando.soria@uam.es}


\thanks{ All authors were partially supported by Ministerio de Economia y Competitividad under grants MTM2013-40846-P and MTM2016-80474-P (Spain).  L. M. was also supported  by  Gruppo Nazionale per l'Analisi Matematica, la Probabilit\`a e le loro Applicazioni (GNAMPA) of the Istituto Nazionale di Alta Matematica (INdAM)}

\begin{document}

\begin{abstract}
In this paper we study a variational Neumann problem for the higher order $s$-fractional Laplacian, with $s>1$. In the process we introduce some non local Neumann boundary conditions that appear in a natural way from a Gauss-like integration formula.
\end{abstract}

\maketitle


\medskip
\section{Introduction and results}\label{introdu}

In this paper we introduce a natural Neumann problem for the higher-order fractional Laplacian $(-\Delta)^su$, $s>1$.

Let us recall that when $0<s<1$ the operator is usually defined, for smooth  functions, by means of the following principal value
\begin{equation}\label{defoperador}
(-\Delta)^s u(x):=c_{N,s}\,P.V.\,\int_{\mathbb{R}^N}\frac{u(x)-u   (y)}{|x-y|^{N+2s}}\,dy,\quad 0<s<1.
\end{equation}
Here,
\begin{equation}\label{eq:cnsig}
c_{N,s}:=\left(\int_{\mathbb{R}^N}\frac{1-\cos(\xi_1)}{|\xi|^{N+2s}}\,d\xi\right)^{-1}=2^{2s-1}\pi^{-\frac{N}{2}}\frac{\Gamma(\frac{N+2s}{2})}{|\Gamma(-s)|},\end{equation}
is a normalized constant. See for example  \cite{BMS, DnPV, DMPS}. It is well-know that for functions, say, in the Schwartz class  $\mathcal{S}(\mathbb{R}^N)$ this operator has an equivalent definition via the Fourier transform that is also valid when $s>1$. More precisely,
\begin{equation}\label{eq:sajhhajhsj}
\widehat{(-\Delta)^{s}}u(\xi)=|\xi|^{2s}\widehat{u}(\xi),\quad \xi\in\mathbb{R}^{N},\, s>0.
\end{equation}
From now on, for the sake of simplicity we will consider here the higher order fractional Laplacian with $s=1+\sigma$, $0<\sigma<1$, so that $s\in(1,2)$. Following the expression given in \eqref{defoperador}, in this case for $u$ smooth, we can also define the operator as
\begin{equation}\label{eq:defoper}
(-\Delta)^{s}u(x)=c_{N,\sigma}\, P.V.\int_{\mathbb R^N}\frac{(-\Delta u(x))-(-\Delta u(y))}{|x-y|^{N+2\sigma}}\,dy,\quad 1<s<2,
\end{equation}
where $c_{N,\sigma}$ is the nomalization constant given in \eqref{eq:cnsig}.  If $0<s<1$ there are many results  regarding  existence, regularity and qualitative properties of solutions of nonlocal problems that involve the operator $(-\Delta)^s$  (see \cite{BDGQ, BDGQ2, DSV, G0, RoS0, RoS1} and the references therein; this list of publications is far from being complete).
 The study of the non local higher order  operator, compared to the better understood lower order non local  operator (i.e. $s\in (0,1)$) has not been entirely developed yet.

 In the higher order case, for example, the lack  in general of a  maximum principle  introduces some new  difficulties.  Some  results on  this subject,  like existence and representation  of solutions, integration by parts, regularity, best Sobolev constants,  maximum principles, Pohozaev identities and  spectral results among others   can be found in   the list of papers~\cite{AJS, CT, DG, G0, G1, MYZ, RoS,Y} or in the corresponding  bibliography of each of them.

\

For what concerns the Neumann problem for the  fractional Laplacian $(-\Delta)^s$, in the case $s\in (0,1)$  and in other similar $s$-nonlocal operators, different approaches have been developed in the literature; see for instance \cite{BCGJ, BGJ, BBC, CK, CERW, CERW1, CERW2, DRoV, G,  MPV, SV}. The reader may find a comparison between some of these different models in  \cite{DRoV}. We also notice here that all the Neumann conditions presented in the previous works regarding  $(-\Delta)^s$, $0<s<1$, are easily seen to approach the classical one ${\partial_\nu u}$ when $s\to 1$.  Nevertheless the one presented in \cite{DRoV} by S. Dipierro, X. Ros-Oton and E. Valdinoci allows us to work in a variational framework and, as the authors describe in Section 2 of the aforementioned paper \cite{DRoV}, it also has a natural probabilistic interpretation. To be more precise, the authors introduce and study the existence and uniqueness of solutions of the following  Neumann problem for the fractional ($s\in (0,1) $) Laplacian
\begin{equation}\label{eq:ESX}
\begin{cases}
(-\Delta)^s u = f(x) &\text{in}\,\,\Omega\\
\mathcal N_s u  =g &\text{in}\,\,\mathbb{R}^N\setminus\overline\Omega,
\end{cases}\end{equation}
where $f,g$ are appropriate problem data. Here, the operator $\mathcal N_s v$ denotes the nonlocal normal derivative defined, for smooth functions, by
\begin{equation}\label{eq:neumannenrico}
\mathcal N_s v(x):=c_{N,s}\int_{\Omega}\frac{v(x)-v(y)}{|x-y|^{N+2s}}\,dy, \qquad x\in \mathbb{R}^N\setminus\overline \Omega.
\end{equation}
This condition can be seen as the natural one  to have the associated  Gauss and Green formulas that allow to use a variational approach in the analysis of problem \eqref{eq:ESX} similar to the local Neumann problem $-\Delta u = f(x)$ in $\Omega$, with $\partial_{\nu} u= g $ on $\partial \Omega$.

\

In the case of higher order operators even in the local case the situation is more involved in general as one can see in, for example, \cite{LA, BL,  C1, V}. In particular in \cite{LA}, by using a {\em Biharmonic Green Formula}, the authors define the Neumann problem for the biharmonic operator $\Delta ^2 u$ and the natural boundary Neumann that, in dimension $N=2$, {rises in the study of the bending of free plates.} As far as we know the problem of establishing a reasonable Neumann condition asociated to $(-\Delta)^su$, $s>1$ has not been developed yet. Therefore, the aim of this work is to introduce  a Neumann problem for the higher order fractional $s$-Laplacian, $s\in (1,2)$, and to
 study the problem
$$
\begin{cases}
(-\Delta)^s u = f(x) &\text{in}\,\,\Omega,\quad 1<s<2\\
\mbox{\it{s-Neumann conditions }} u  =g&\text{in}\,\,\mathbb{R}^N\setminus\overline{\Omega}.
\end{cases}
$$
Here, and throughout the paper, $\Omega$ denotes a smooth bounded domain and our approach is to look for a variational formulation of the problem. Using a similar integration by parts as in the lower order case, $0<s<1$, we can see that for a smooth function  $u$ one has
\begin{equation}\nonumber
\int_{\Omega}(-\Delta)^su\,dx=-\int_{\mathbb R^N\setminus \Omega} \mathcal N_\sigma (-\Delta u) \,dx,
\end{equation}
where
$$\mathcal N_\sigma(-\Delta u)(x)=(-\Delta)_{\Omega}^\sigma (-\Delta u)(x)=\int_{\Omega}\frac{{(-\Delta u)(x)-(-\Delta u)(y)}}{|x-y|^{N+2\sigma}}dy,\, x\in\mathbb{R}^{N}\setminus\overline{\Omega}.$$
However, in order to obtain a Green formula seeking  a variational formulation of the problem, it will be necessary to split this last condition in two parts. Following this path and via a {\em non local Green Formula type},  we are lead to define two non local operators $\mathcal N^1_\sigma, \mathcal N^2_\sigma $, that will play the role of the {\it s-Neumann conditions} for our problem. More precisely, we will study the following
\begin{equation}\label{eq:p}\tag{$\mathcal P$}
\begin{cases}
(-\Delta)^s u = f(x) &\text{in}\,\,\Omega, \quad 1<s<2\\
\mathcal N_\sigma^1 u  =g_1&\text{in}\,\,\mathbb{R}^N\setminus\overline\Omega=: \mathcal C\overline\Omega\\
\mathcal N_\sigma^2 u=g_2 &\text{on}\,\,\partial\Omega,
\end{cases}\end{equation}
where $f$, $g_1$ and $g_2$ satisfy some  suitable hypotheses that we will specify  below and {$\Omega\subset\mathbb R^N$ be a bounded $C^{1,1}$ domain (unless we specify something different as, for example, in Lemma \ref{regional} below)}. The definition of the  operators $\mathcal{N}_{\sigma}^{i}$, $i=1,\, 2$ for suitable  $v\in \mathcal{S}(\mathbb{R}^{N})$  will come in a natural way from the  integration by parts formula stated below in Theorem \ref{pro:intbyparts2} as follows
\begin{equation}\label{eq:neumann1}
\mathcal N^1_\sigma v(x){:=-\operatorname{div}(-\Delta)^{\sigma}_{\Omega}(\nabla v)}(x),  \qquad x\in \mathbb{R}^N\setminus\overline \Omega,
\end{equation}
and
\begin{equation}\label{eq:neumann2}
\mathcal N^2_\sigma v(x){:=(-\Delta)^{\sigma}_{\mathbb{R}^N\setminus \Omega}(\nabla v)(x)\cdot \nu},  \qquad x\in \partial \Omega,
\end{equation}
where  $\nu$ is the outer unit normal field to $\partial \Omega$. Also, $(-\Delta)^{\sigma}_{\mathcal{A}}w$ denotes the regional fractional Laplacian that, for an open set $\mathcal A\subset \mathbb R^N$ and  regular functions $w$,  is defined  by
\begin{equation}\label{eq:numebehebqfgwdg}
(-\Delta)^\sigma_{\mathcal A} w (x):=c_{N,\sigma}\,\lim_{\varepsilon \rightarrow 0^+}\,\int_{\mathcal A\setminus  B_\varepsilon(x)}\frac{w(x)-w(y)}{|x-y|^{N+2\sigma}}\,dy,\quad {x\in\mathbb{R}^{N}}\setminus \partial\mathcal A,
\end{equation} where  $c_{N,\sigma}$ is defined in \eqref{eq:cnsig}. 

{We give now some remarks about the regional operator (see also \cite{Mou} and the references therein)}. First of all we notice that, as we will see, the operator may not be {pointwise well defined} for $x\in \partial\mathcal A$. {For a detailed explanation under which conditions the pointwise definition up to the boundary can be considered see, for instace, \cite[Theorem 5.3]{Q-YM}.} Nonetheless, we observe that the principal value in the previous definition is not needed when $x\in \mathbb{R}^N\setminus\overline{\mathcal A}$ if $w$  is sufficiently regular, say for instance $w\in\mathcal{S}(\mathbb{R}^{N})$. The same is true if $x\in\mathcal{A}$ and $\sigma<1/2$. However if $x\in\mathcal A$ and  $\sigma\geq 1/2$, even if $w\in\mathcal{S}(\mathbb{R}^{N})$, the principal value is required. In fact, if $x\in\mathcal A$ denoting by $\rho(x)=dist(x,\partial\mathcal A)$ and  $B_x=B_{\rho(x)}(x)$ then
\begin{equation}\label{drexler}
(-\Delta)_{\mathcal A}^\sigma w(x)= c_{N,\sigma}\left(\int_{\mathcal A\setminus B_x}\frac{w(x)-w(y)}{|x-y|^{N+2\sigma}} dy+
\lim_{\epsilon\to 0}\int_{B_x\setminus B_\epsilon(x)}\frac{w(x)-w(y)}{|x-y|^{N+2\sigma}}dy\right).
\end{equation}
Using now that $B_x\setminus B_\epsilon(x)$ is a symmetric domain around $x$ it follows that
$$\lim_{\epsilon\to 0}\int_{B_x\setminus B_\epsilon(x)}\frac{w(x)-w(y)}{|x-y|^{N+2\sigma}}dy=
\lim_{\epsilon\to 0}\int_{B_x\setminus B_\epsilon(x)}\frac{w(x)-w(y)-\nabla w(x)(x-y)}{|x-y|^{N+2\sigma}}dy.$$
Since the previous integral is absolutely convergent for example if $w\in \mathcal{C}^{1,1}$, from \eqref{drexler} we get that, when $\sigma\geq 1/2$,
\begin{equation}\label{drexler2}
(-\Delta)_{\mathcal A}^\sigma w(x)=c_{N,\sigma}
\begin{cases}
\displaystyle \int_{\mathcal A\setminus B_x}\frac{w(x)-w(y)}{|x-y|^{N+2\sigma}} dy+\int_{B_x}\frac{w(x)-w(y)-\nabla w(x)(x-y)}{|x-y|^{N+2\sigma}}dy,&\text{if}\,\,x\in \mathcal A,\\ \\
\displaystyle\int_{\mathcal A} \frac{w(x)-w(y)}{|x-y|^{N+2\sigma}}dy,&\text{if}\,\,x\in\mathbb{R}^N\setminus\overline{\mathcal A}.
\end{cases}
\end{equation}
{Nevertheless, according to Theorem B in \cite{Mou}  the operator defined by \eqref{eq:neumann2} can be undestood in the trace sense. In this way will be considered hereafter.}

\

Before announcing the main result of this work we introduce the following notation and definitions:
\begin{definition}\label{def:affinfunctions}
By  $\mathcal P_{1}(\mathbb R^N)$ we denote the vector space of all  polynomials of degree one  with real coefficients, that is,
\[\mathcal P_{1}(\mathbb R^N)=\left\{p(x):\mathbb R^N\rightarrow \mathbb R\,|\ p(x)= c_0+ (c,x),\,\, \text{with}\,\,   c_0\in \mathbb R\,\,   \text{and}\,\,c,x \in \mathbb R^N\right\},\]
where  $(\cdot, \cdot):\mathbb R^N\times \mathbb R^N \rightarrow \mathbb R$ represents the Euclidean scalar product in $\mathbb R^N$.
\end{definition}

We define also  $\dot{H}^s(\Omega)$ as the class of functions given by
\begin{equation}\label{spacebase}
\dot{H}^s(\Omega)=\left\{\, u:\mathbb{R}^N\rightarrow\mathbb{R}\,|\, u \hbox{ weakly differentiable, so that  } \, D(u)<\infty \right\},\,\,
\end{equation}
where
\begin{equation*}
D(u):=\sqrt{
\iint_{Q(\Omega)} \dfrac{|\nabla u(x)-\nabla u(y)|^2}{|x-y|^{N+2\sigma}}dxdy},
\end{equation*}
and
\begin{equation}\nonumber
Q(\Omega):=\mathbb{R}^{2N}\setminus(\mathbb{R}^N\setminus \Omega)^2.
\end{equation}
Notice that $\mathcal P_{1}(\mathbb R^N)\subset \dot{H}^s(\Omega)$.

Next we will define the class of  admissible data.

Let $g_1\in {L^1( \mathbb{R}^N\setminus\Omega, |x|^2\, dx )}{\cap L^1(\mathbb{R}^N\setminus\Omega)}$.  Associated to $g_1$ we consider the positive measure in $\mathbb{R}^N$, absolutely continuous with respect to Lebesgue measure, defined by
\begin{equation}\label{dmu}
d\mu_{g_1}=(\chi_\Omega+|g_1|\chi_{\mathbb{R}^N\setminus\Omega})dx,
\end{equation}
and the class of functions
\begin{equation}\label{ortogonal}
H^{s,0}_{g_1}(\Omega):=\Big\{u\in \dot{H}^s(\Omega):\, \int_{\mathbb{R}^{N}}{u p\, d\mu_{g_1}}=0,\, {\forall} p\in\mathcal{P}_{1}(\mathbb{R}^{N})\Big\}.
\end{equation}
For the associated measure $d\mu_{g_1}$ we consider the following Rayleigh quotient
\begin{equation}\label{espectral}
\lambda_1(g_1)=\inf_{u\in H^{s,0}_{g_1}, \, u\ne 0}\dfrac{\dyle\int\int_{Q(\Omega)} \dfrac{|\nabla u(x)-\nabla u(y)|^2}{|x-y|^{N+2\sigma}}dxdy}{\dyle\int_{\mathbb{R}^N}u^2d\mu_{g_1}}.
\end{equation}
\begin{definition}\label{admissible}\textit{($(\mathcal A_{(f,g_1,g_2)})$ assumptions).}

 We say that $(f,g_1,g_2)$ is an admissible data triplet if
\begin{enumerate}
\item $f\in L^2(\Omega)$
\item $g_1\in {L^1( \mathbb{R}^N\setminus\Omega, |x|^2\, dx )}{\cap L^1(\mathbb{R}^N\setminus\Omega)}$ and the corresponding measure
$d\mu_{g_1}$ satisfy that  the spectral value $\lambda_1(g_1)$ defined  by \eqref{espectral} is strictly positive.
\item $g_2\in L^2(\partial \Omega)$.
\end{enumerate}
\end{definition}

As a direct consequence of the definition, given an admissible $g_1$ we have that
$$\int_{\mathbb{R}^N}u^2d\mu_{g_1}<+\infty, \hbox{ for all } u\in  H^{s,0}_{g_1}(\Omega).$$
Also, by the hypotheses on integrability of $g_1$, one has
$$\int_{\mathbb{R}^N}p^2d\mu_{g_1}<+\infty, \hbox{ for all } p\in \mathcal P_{1}(\mathbb R^N).$$

\begin{example}
Every function  $g_1$  such that $g_1\in L^p(\mathbb{R}^N\setminus\Omega)$ with $p>\frac N2$ and with compact support satisfies condition $(2)$ in  Definition \ref{admissible} (see Lemma \ref{equivalencia} below).
\end{example}
Now we are ready to state  the main result of the paper:

\begin{theorem}\label{main}
Let $\Omega\subset\mathbb R^N$ be a bounded $C^{1,1}$ domain and let us suppose that the assumptions $(\mathcal A_{(f,g_1,g_2)})$ hold.
Then, the problem \eqref{eq:p} has a weak solution (in the sense of Definition \ref{def:weaksolution}), if and only if the following compatibility condition hold{s}
\begin{equation}\label{eq:unodiez}
\int_{\Omega} f p \, dx +\int_{\mathbb R^N\setminus \Omega}g_1p \, dx+\int_{\partial \Omega}g_2 p\, dS=0, \qquad \text{for all }\,\,  p\in \mathcal P_1(\mathbb R^N).
\end{equation}
Moreover, if \eqref{eq:unodiez} holds,  the solution is unique up to an affine function $p\in \mathcal P_1(\mathbb R^N).$
\end{theorem}


\

The paper is organized as follows: in Section 2 we present the integration by parts formula that shows the key point in order to understand the variational structure of the problem \eqref{eq:p}. In Section 3 we give some preliminaries related to the functional framework associated to problem \eqref{eq:p} and we introduce the proper notion of solution that will be used along this work. Section 4 deals with the proof of Theorem \ref{main}. In Section 5 we give the complete description of the structure of the eigenvalues and eigenfunctions of \eqref{eq:p}. Finally, in Section 6 we briefly comment other problems and results related with the one studied here.

\

Throughout the paper, generic fixed numerical constants will be denoted by $C$, in some cases with a subscript and/or a superscript, and will be allowed to vary within a single line or formula.
\section{Computations in  $\mathbb{R}^N$ and a motivation of the problem \eqref{eq:p}}
The main objective of this section is to prove a new integration by parts formula associated to $(-\Delta)^s$, $1<s<2$. In the sequel by $ Q(\Omega)$, we mean
\begin{equation}\nonumber
Q(\Omega):=\mathbb{R}^{2N}\setminus(\mathbb{R}^N\setminus \Omega)^2.
\end{equation}
First of all we need the following result that allows us to write the fractional operator in a divergence form.
\begin{proposition}\label{pro:equivfraclapl} Given $u\in \mathcal \mathcal{S}(\mathbb{R}^{N}) $ and $s=m+\sigma$ with $m\in\mathbb N$ and $\sigma \in (0,1).$
The operator $(-\Delta)^s u$ can be expressed  in one of the following ways
\begin{equation}\label{eq:ksakjasjkajahdshhiohio}
(-\Delta)^su=-\Delta^m\big((-\Delta)^\sigma u\big)=(-\Delta)^\sigma \big(-\Delta^m u\big)= -\operatorname{div}(-\Delta)^{\frac{m-1}{2}}(-\Delta)^\sigma (-\Delta)^{\frac{m-1}{2}} \nabla u,
\end{equation}
if $m$ is odd, or
\begin{equation}\nonumber
(-\Delta)^su=-\Delta^m\big((-\Delta)^\sigma u\big)=(-\Delta)^\sigma \big(-\Delta^m u\big)= (-\Delta)^{\frac{m}{2}}(-\Delta)^\sigma (-\Delta)^{\frac{m}{2}} u,
\end{equation}
if $m$ is even.
\end{proposition}
\begin{proof}
It is sufficient to use the Fourier transform $\mathscr F (\cdot)$ and the multiplicative semigroup property. We prove the last equality in \eqref{eq:ksakjasjkajahdshhiohio}. The others follow in the same way.
\begin{equation}\label{eq:gunsandroses}
 -\operatorname{div}(-\Delta)^{\frac{m-1}{2}}(-\Delta)^\sigma (-\Delta)^{\frac{m-1}{2}} \nabla u=-\sum_{j=1}^N\partial_j(-\Delta)^{\frac{m-1}{2}}(-\Delta)^\sigma (-\Delta)^{\frac{m-1}{2}} \partial_j u.
\end{equation}
Using Fourier transform in \eqref{eq:gunsandroses}, we obtain
\begin{eqnarray}\label{eq:sakdhjdhiuiu}
&&\mathscr F\big (-\operatorname{div}(-\Delta)^{\frac{m-1}{2}}(-\Delta)^\sigma (-\Delta)^{\frac{m-1}{2}} \nabla u \big)= - \sum_{j=1}^N (i\xi_j)|\xi|^{(m-1)}|\xi|^{2\sigma}|\xi|^{(m-1)}(i\xi_j) \mathscr F u\\\nonumber
&&=|\xi|^{2(m+\sigma)}\mathscr F u
\end{eqnarray}
Recalling \eqref{eq:sajhhajhsj}, from \eqref{eq:sakdhjdhiuiu} we deduce
\begin{equation}\nonumber
(-\Delta)^su:=\mathscr F^{-1}\Big(|\xi|^{2(m+\sigma)}\mathscr F u\Big)= -\operatorname{div}(-\Delta)^{\frac{m-1}{2}}(-\Delta)^\sigma (-\Delta)^{\frac{m-1}{2}} \nabla u.
\end{equation}
\end{proof}
\subsection{Integration by parts formula}
In this section we  prove different integration  formulas that  will be essential to define a  variational formulation of the Neumann boundary conditions.

To simplify the  next results, recalling \eqref{eq:numebehebqfgwdg}, for $u\in \mathcal{S}(\mathbb{R}^{N})$ and $\Omega$ {a smooth domain,}  we can write
\[
(-\Delta)^{\sigma}u=(-\Delta)^{\sigma}_{\Omega}u+(-\Delta)^{\sigma}_{\mathbb{R}^N\setminus \overline{\Omega}}u, \quad \mbox{a.e.}
\]
The operators $(-\Delta)^{\sigma}_{\Omega}u$  and    $ (-\Delta)^{\sigma}_{\mathbb{R}^N\setminus \overline{\Omega}}u$ are the regional $\sigma$-Laplacian for $\Omega$ and $\mathbb{R}^N\setminus \overline{\Omega}$ respectively.   We refer for instante to \cite{DnPV}, \cite{Guan}, \cite{Mou} and the references therein for  the properties of the regional fractional laplacian.

For the reader convenience we include the following result that will be used in the next calculations.
\begin{lemma}\label{regional} Let $\Omega$ be a $\mathcal{C}^{1,1}$  domain that could be unbounded such that its boundary, $\partial\Omega$, is a compact set.
Then for all $u\in \mathcal{S}(\mathbb{R}^{N})$,
$$\int_\Omega (-\Delta)^\sigma_\Omega u(x) dx=0.$$
\end{lemma}
\begin{proof} Assume that $\Omega$ is bounded; if $0<\sigma<\frac 12$ the result is obvious given that the function
$$G(x,y)=\frac{u(x)-u   (y)}{|x-y|^{N+2\sigma}}\in L^1(\Omega\times\Omega),$$
and $G(x,y)=-G(y,x)$.

We consider now  the case $\frac 12\le \sigma<1$ in which the principal value is present.
Consider,
$$f_\epsilon(x)= \int_{\Omega\setminus B_\epsilon(x)}\frac{u(x)-u   (y)}{|x-y|^{N+2\sigma}}dy,\, \quad  \epsilon>0,\,  x\in\Omega.$$
If we are able to find $h(x)\in L^1(\Omega)$ such that $|f_\epsilon(x)|\le h(x)$, $x\in \Omega$ then the result follows by the Dominated Convergence Theorem; indeed
$$\int_\Omega (-\Delta)^\sigma_\Omega u(x) dx=\int_\Omega \lim_{\epsilon \to 0} f_\epsilon(x)dx=\lim_{\epsilon \to 0}\iint_{\Omega\times\Omega\setminus\{(x,y)\,|\, |x-y|<\epsilon\}}\left(\frac{u(x)-u   (y)}{|x-y|^{N+2\sigma}}\right)dxdy=0$$
by the antisymmetry, as above.

To find a function $h(x)\in L^1(\Omega)$ majoring the $f_\epsilon(x)$ family, fix $x\in \Omega$.
Define $\rho(x)=dist(x,\partial\Omega)$,  $B_x=B_{\rho(x)}(x)$ and consider first the case $0<\epsilon<\rho(x)$. Then
$$f_\epsilon(x)=\int_{\Omega\setminus B_x}\frac{u(x)-u (y)}{|x-y|^{N+2\sigma}} dy+
\int_{B_x\setminus B_\epsilon(x)}\frac{u(x)-u(y)}{|x-y|^{N+2\sigma}}dy.$$
Now by antisymmetry we     find that
$$\left|\int_{B_x\setminus B_\epsilon(x)}\frac{u(x)-u(y)}{|x-y|^{N+2\sigma}}dy\right| \le
\int_{B_x}\left|\frac{u(x)-u(y)-\nabla u(x)(x-y)}{|x-y|^{N+2\sigma}}\right|dy$$
where the last term has a quadratic cancelation and becomes a term in $L^1(\Omega)$. Finally,  we estimate the first term  as follows. Take $R=2\mbox{diam}(\Omega)$
$$\left|\int_{\Omega\setminus B_x}\frac{u(x)-u (y)}{|x-y|^{N+2\sigma}} dy\right|\le \int_{\Omega\setminus B_x}\left|\frac{u(x)-u (y)}{|x-y|^{N+2\sigma}}\right| dy
\le C_1 \int_{B_R(x)\setminus B_x}\frac{dy}{|x-y|^{N+2\sigma-1}}\le C_2 \int_{\rho(x)}^R \frac{dt}{t^{2\sigma}}. $$
The case $ \epsilon\ge\rho(x)$ is simpler since then $\Omega\setminus B_x\supset \Omega\setminus B_\epsilon(x)$.
Summarizing,
$$|f_\epsilon(x)|\le\begin{cases}
 O(1),\quad  0<\sigma<\frac 12\\
 -\log \rho(x)+O(1),\quad \sigma=\frac 12\\
 \dfrac 1{\rho(x)^{2\sigma-1}}+O(1), \quad \frac 12<\sigma <1.
\end{cases}
$$
If $\Omega$ is unbounded, inside of a ball containing the boundary we reproduce the same calculations that in the bounded case and outside we take into account   the decay od the kernel, that is
$$|(-\Delta)_\Omega u(x)|\le \frac{C}{|x|^{N+2\sigma}}.$$
Then we apply again the Dominated Convergence Theorem to conclude.
\end{proof}

Now we can establish the following
\begin{proposition}\label{pro:intbyparts1}Let $u\in \mathcal{S}(\mathbb{R}^{N})$, $s=1+\sigma$ and $\Omega\subseteq\mathbb{R}^{N}$ be a smooth domain, possibly unbounded, with compact boundary. Then
\begin{equation}\nonumber
\int_{\Omega}(-\Delta)^su\,dx=-\int_{\mathbb R^N\setminus \Omega} \mathcal N_\sigma (-\Delta u) \,dx,
\end{equation}
where
$$\mathcal N_\sigma(-\Delta u)(x)=(-\Delta)_{\Omega}^\sigma (-\Delta u)(x)=\int_{\Omega}\frac{{(-\Delta u(x))-(-\Delta u(y))}}{|x-y|^{N+2\sigma}}dy,\, x\in \mathbb R^N\setminus \overline{\Omega}.$$
\end{proposition}
\begin{proof} For $u\in \mathcal{S}(\mathbb{R}^{N})$ we note that $(-\Delta)^s u$ is well defined in all $\mathbb R^N$ and actually there exists a positive constant $C=C(N, \sigma,\|\Delta u\|_{L^{\infty}(\mathbb{R}^N)})$, such that $|(-\Delta)^s u|\leq C$. By direct computations we obtain
\begin{equation}\label{eq:c1}
\begin{array}{lll}
&\dyle \int_\Omega (-\Delta)^s u\,dx=c_{N,\sigma}\int_\Omega P. V. \int_{\mathbb R^N}\frac{(-\Delta u(x))-(-\Delta u(y))}{|x-y|^{N+2\sigma}}\,dy\,dx\\
&\dyle =c_{N,\sigma}\int_\Omega \int_{\mathbb R^N\setminus \Omega}\frac{(-\Delta u(x))-(-\Delta u(y))}{|x-y|^{N+2\sigma}}\,dy\,dx\\
&\dyle =c_{N,\sigma}\int_{\mathbb R^N\setminus \Omega}\int_\Omega \frac{(-\Delta u(x))-(-\Delta u(y))}{|x-y|^{N+2\sigma}}\,dx\,dy\\
&\dyle=-\int_{\mathbb R^N\setminus \Omega} \mathcal N_\sigma (-\Delta u)\,dy,
\end{array}
\end{equation}
where in \eqref{eq:c1} we have use  Lemma \ref{regional} that gives
$$
\int_\Omega P. V. \int_\Omega \frac{(-\Delta u(x))-(-\Delta u(y))}{|x-y|^{N+2\sigma}}\,dy\,dx=0.
$$

\end{proof}

We now show some {\it calculation rules} that will be needed later.
\begin{lemma}\label{lem:primolemmaintperpar}Let  $u\in \mathcal{S}(\mathbb{R}^{N})$ and $\Omega\subseteq\mathbb{R}^{N}$ be a smooth domain with compact boundary.  Then, for every $0<\sigma<1$ we have
\begin{itemize}
\item [$(i)$]
\begin{equation*}
\int_{\mathbb{R}^N\setminus \Omega} (-\Delta)^{\sigma}_{\Omega}u \,dx=-\int_{\Omega} (-\Delta)^{\sigma}_{\mathbb{R}^N\setminus \Omega}u \,dx,
\end{equation*}
\item [$(ii)$]
\begin{equation*}
\int_{\Omega}(-\Delta)^{s}u=-\int_{\mathbb{R}^N\setminus \Omega}(-\Delta)^{s}u,\quad s=1+\sigma.
\end{equation*}
\end{itemize}
\end{lemma}
\begin{proof}
To prove $(i)$ it is sufficient to apply Fubini's theorem and $(ii)$ follows by using Proposition \ref{pro:intbyparts1} and  Lemma \ref{regional}.
\end{proof}
Thanks to Lemma \ref{lem:primolemmaintperpar} we have the following result that will be needed to prove the main theorem of the present work.
\begin{proposition}\label{pro:cuatrodiec}
Let $p\in \mathcal P_1(\mathbb R^N) $ and let $u\in \mathcal{S}(\mathbb{R}^{N})$ be such that
\begin{equation}\label{eq:sagaghagdvjhgsjd}
\int_{\mathbb{R}^N\setminus \Omega}|\operatorname{div}\big((-\Delta)_{ \Omega}^\sigma \nabla u(x)\big)\, p(x)|\,dx<+\infty.
\end{equation}
Then
\begin{equation*}
\int_{\Omega}p\, (-\Delta)^su \,dx=
-\int_{\mathbb{R}^N\setminus \Omega} p\, \mathcal N^1_\sigma u \,dx -\int_{\partial \Omega}  p\, \mathcal{N}^2_\sigma u    \,d S.
\end{equation*}
\end{proposition}
\begin{proof}
If $p\in \mathcal P_1(\mathbb R^N) $ and $u\in \mathcal{S}(\mathbb{R}^{N})$ then
\begin{eqnarray}\label{eq:chenonsipuoooo}
&&\int_{\Omega}p\, (-\Delta)^{s}u\,dx= \int_{\Omega}-\operatorname{div}\big((-\Delta)^\sigma \nabla u\big)\,p\,dx\\\nonumber
&&= \int_{\Omega}-\operatorname{div}\big((-\Delta)_{\Omega}^\sigma \nabla u\big)\,p\,dx+ \int_{\Omega}-\operatorname{div}\big((-\Delta)_{\mathbb{R}^N\setminus \Omega}^\sigma \nabla u\big)\,p\,dx\\\nonumber
&&=: I_1+I_2.
\end{eqnarray}
By the divergence theorem we have that
\begin{eqnarray}\nonumber
I_1&=&\int_{\Omega}-\operatorname{div}\big((-\Delta)_{\Omega}^\sigma \nabla u\big)\,p\,dx\\\nonumber
&=& \int_{\Omega} (-\Delta)_{\Omega}^\sigma \nabla u \cdot \nabla p\, dx- \int_{\partial \Omega}p\, (-\Delta)_{\Omega}^\sigma \nabla u \cdot \nu \,dS,
\end{eqnarray}
where $\nu$ denotes the unit outer normal field to the boundary $\partial \Omega$. Since $\nabla p$ is a constant vector, then using $(i)$ of Lemma \ref{lem:primolemmaintperpar}, we obtain that
\begin{equation}\label{eq:kjsakajjjhkwdjhkwefw}
I_1=-\int_{\partial \Omega}p\, (-\Delta)_{\Omega}^\sigma \nabla u \cdot \nu \,dS.
\end{equation}
By divergence theorem
\begin{eqnarray}\nonumber
I_2&:=& \int_{\Omega}-\operatorname{div}\big((-\Delta)_{\mathbb{R}^N\setminus \Omega}^\sigma \nabla u\big)\,p\,dx\\\nonumber
&=& \int_{\Omega} (-\Delta)_{\mathbb{R}^N\setminus \Omega}^\sigma \nabla u \cdot \nabla p\, dx- \int_{\partial \Omega}p\, (-\Delta)_{\mathbb{R}^N\setminus \Omega}^\sigma \nabla u \cdot \nu \,dS.
\end{eqnarray}
Recalling that $\nabla p$ is a constant vector, using $(i)$ of Lemma \ref{lem:primolemmaintperpar}, we obtain
\begin{equation}\label{eq:sajhsajhhsinceraaa}
I_2=-\int_{\mathcal C \Omega}(-\Delta)^{\sigma}_{\Omega}\nabla u\cdot \nabla p\, dx - \int_{\partial \Omega}p\, (-\Delta)_{\mathbb{R}^N\setminus \Omega}^\sigma \nabla u \cdot \nu \,dS.
\end{equation}
Using \eqref{eq:kjsakajjjhkwdjhkwefw} and  \eqref{eq:sajhsajhhsinceraaa} together with divergence theorem,  we deduce
\begin{eqnarray}\label{eq:lalunaaaaasefsss}
&&I_1+I_2= -\int_{\partial \Omega}p\, (-\Delta)_{\Omega}^\sigma \nabla u \cdot \nu \,dS\\\nonumber
&&-\int_{\mathcal C \Omega}(-\Delta)^{\sigma}_{\Omega}\nabla u\cdot \nabla p\, dx - \int_{\partial \Omega}p\, (-\Delta)_{\mathbb{R}^N\setminus \Omega}^\sigma \nabla u \cdot \nu \,dS\\\nonumber
&&=-\int_{\partial \Omega}p\, (-\Delta)_{\Omega}^\sigma \nabla u \cdot \nu \,dS \\\nonumber
&&+\int_{\mathbb{R}^N\setminus \Omega}\operatorname{div}\big((-\Delta)_{ \Omega}^\sigma \nabla u\big)\,p\,dx-\int_{\partial \Omega}p\, (-\Delta)_{ \Omega}^\sigma \nabla u \cdot (-\nu)\, ds
-\int_{\partial \Omega}p\, (-\Delta)_{\mathbb{R}^N\setminus \Omega}^\sigma \nabla u \cdot \nu \,dS\\\nonumber
&&=-\int_{\mathbb{R}^N\setminus \Omega}-\operatorname{div}\big((-\Delta)_{ \Omega}^\sigma \nabla u\big)\,p\,dx-\int_{\partial \Omega}p\, (-\Delta)_{\mathbb{R}^N\setminus \Omega}^\sigma \nabla u \cdot \nu \,dS,
 \end{eqnarray}
where $(-\nu)$ denotes the unit inner normal field to the boundary $\partial \Omega$. We point out that in the previous computations,  the divergence theorem ({see for example \cite[Theorem 6.3.4]{Will}}) can be used using a truncation argument together with \eqref{eq:sagaghagdvjhgsjd}.

Collecting \eqref{eq:chenonsipuoooo} and \eqref{eq:lalunaaaaasefsss}, using the definitions \eqref{eq:neumann1} and \eqref{eq:neumann2}, we conclude the proof.
\end{proof}
\begin{remark}
We notice that from Proposition \ref{pro:intbyparts1} and Proposition \ref{pro:cuatrodiec} it is clear that
\begin{equation}\label{eq:chetempochefa}
\int_{\mathbb{R}^N\setminus \Omega} \mathcal N_\sigma (-\Delta u)\,dx=\int_{\mathbb{R}^N\setminus \Omega} \mathcal N^1_\sigma u\,dx +\int_{\partial \Omega} \mathcal{N}^2_\sigma u  \,d S,\, \mbox{ for every $u\in\mathcal{S}(\mathbb{R}^{N})$ }
\end{equation}
in the hypotheses of Proposition \ref{pro:cuatrodiec},
that is, \eqref{eq:chetempochefa} is the splitting of $\mathcal N_\sigma (-\Delta u)$ in the two parts that will be needed for a variational formulation of the corresponding Neumann problem.
\end{remark}
We conclude this section  obtaining a natural Neumann condition for the s-Laplacian with $s>1$.   Roughly speaking, in the higher order case, to describe an  appropriate weak formulation of our problem, we have to use   two (non local) Neumann conditions. Our candidates are given in  equations \eqref{eq:neumann1} and \eqref{eq:neumann2}. Thus, although Proposition \ref{pro:intbyparts1}  suggests the use of $\mathcal N_\sigma (-\Delta u)$ as the  Neumann condition for problem \eqref{eq:p} we rather  split it    into $\mathcal N_\sigma^1 u$ and  $\mathcal N_\sigma^2 u$ via the equation \eqref{eq:chetempochefa}. The fact that this is the right splitting follows from the following proposition.

\begin{proposition}\label{pro:intbyparts2}
Let $u \in \mathcal{S}(\mathbb{R}^{N})$ be  such that
\begin{equation}\label{condition3} \int_{\mathbb{R}^N\setminus \Omega}\left |\operatorname{div}\left((-\Delta)_{ \Omega}^\sigma \nabla u\right)\, \right|\,dx<+\infty,
\end{equation}

and set $s=1+2\sigma$, $0<\sigma<1$. Then, for $v\in \mathcal{S}(\mathbb{R}^{N})$, we have
\begin{eqnarray}\nonumber
&&\frac{c_{N,\sigma}}{2}\int_{Q(\Omega)}\frac{(\nabla u(x)-\nabla u(y))(\nabla v(x)-\nabla v(y))}{|x-y|^{N+2\sigma}}\, dx\,dy\nonumber\\
&&=\int_{\Omega}v\, (-\Delta)^su\,  dx +\int_{\mathbb{R}^N\setminus \Omega} v\, \mathcal {N}^1_\sigma u \, dx +\int_{\partial \Omega} v\, \mathcal{N}^2_\sigma u \, d S.
\end{eqnarray}
\end{proposition}
\begin{proof}
Since $u$ is regular,  a similar argument as in Lemma \ref{regional} shows that
$$\int_{\mathcal A}P.V.\int_{\mathcal A} \frac{(\nabla v(x)+\nabla v(y))(\nabla u(x))-(\nabla u(y))}{|x-y|^{N+2\sigma}}\,dy\,dx=0,$$
for  any open set $\mathcal A\subset \mathbb R^N$. Therefore 
we have
\begin{eqnarray}\label{eq:int1}\\\nonumber&&\frac{1}{2}\int_{Q(\Omega)}\frac{(\nabla u(x)-\nabla u(y))(\nabla v(x)-\nabla v(y))}{|x-y|^{N+2\sigma}}\, dx\,dy = \int_{\Omega}\nabla v(x){P. V.} \int_{\mathbb R^N}\frac{(\nabla u(x)-\nabla u(y))}{|x-y|^{N+2\sigma}}\, dy\,dx \\\nonumber
&&+ \int_{\mathbb{R}^N\setminus \Omega}\nabla v(x)  \int_{\Omega}\frac{(\nabla u(x)-\nabla u(y))}{|x-y|^{N+2\sigma}}\, dy\,dx.
\end{eqnarray}
In each term of the r.h.s of \eqref{eq:int1} we use the divergence theorem. Therefore we  get the following identity
\begin{eqnarray}\label{eq:int11}
&&{c_{N,\sigma}}\int_{\Omega}\nabla v(x) {P. V.}  \int_{\mathbb R^N}\frac{(\nabla u(x)-\nabla u(y))}{|x-y|^{N+2\sigma}}\, dy\,dx \\\nonumber
&&= \int_{\Omega} \left(-\operatorname{div} {\left((-\Delta)^\sigma \nabla u(x)\right)}\right)  v(x)\,dx +\int_{\partial \Omega}v(x) {\left((-\Delta)^\sigma \nabla u(x)\right)} \cdot \nu\, d S,
\end{eqnarray}
where $\nu$ denotes the unit outer normal field to the boundary $\partial \Omega$ and
\begin{eqnarray}\label{eq:int111}
&&{c_{N,\sigma}}\int_{\mathbb{R}^N\setminus \Omega}\nabla v(x) \int_{\Omega}\frac{(\nabla u(x)-\nabla u(y))}{|x-y|^{N+2\sigma}}\, dy\,dx \\\nonumber
&&= \int_{\mathbb{R}^N\setminus \Omega}  \left (-\operatorname{div} {\big((-\Delta)_\Omega^\sigma \nabla u\big)}\right)v(x)\,dx-\int_{\partial \Omega}v(x) {\big((-\Delta)_\Omega^\sigma \nabla u\big)}\cdot \nu\, d S.
\end{eqnarray}
{Thus, by Proposition \ref{pro:equivfraclapl}, putting together \eqref{eq:int11} and \eqref{eq:int111}, from \eqref{eq:int1}} we obtain that
\begin{eqnarray}\label{eq:int1111}\\\nonumber&&\frac{c_{N,\sigma}}{2}\int_{Q(\Omega)}\frac{(\nabla u(x)-\nabla u(y))(\nabla v(x)-\nabla v(y))}{|x-y|^{N+2\sigma}}\, dx\,dy\\\nonumber
&&=
\int_{\Omega}\left( -\operatorname{div}  {\big((-\Delta)^\sigma \nabla u\big)}\right)  v(x)\,dx +\int_{\mathbb{R}^N\setminus \Omega}  \left (-\operatorname{div}  {\big((-\Delta)_\Omega^\sigma \nabla u\big)}\right)v(x)\,dx \\\nonumber &&+\int_{\partial \Omega}v(x) {(-\Delta)_{\mathbb{R}^N\setminus \Omega}^\sigma (\nabla u)}\cdot \nu\, d S\\\nonumber
&&= \int_{\Omega}{v}(-\Delta)^su\, dx +\int_{\mathbb{R}^N\setminus \Omega} {v}\mathcal N^1_\sigma u \,dx +\int_{\partial \Omega} {v}\mathcal N^2_\sigma u  \,d S,
\end{eqnarray}
concluding the proof.
\end{proof}
\subsection{Some considerations about condition \eqref{condition3}} Let us point out here that the integrability condition \eqref{condition3} in Proposition \ref{pro:intbyparts2} is not needed when $0<\sigma<1/2$, for in this case one always {has}
$ \operatorname{div}\left((-\Delta)_{ \Omega}^\sigma (\nabla u)(x)\right)\in L^1(\mathbb R^N\setminus \Omega)$. To see this observe that for  a function $u\in \mathcal{S}(\mathbb{R}^N)$ a simple computation shows that
 \begin{equation}\label{diver}
 \operatorname{div}\left((-\Delta)_{ \Omega}^\sigma (\nabla u)(x)\right)=
 c_{N,\sigma}\int_\Omega \frac{\Delta u(x)-(N+2\sigma)\frac{(\nabla u(x)-\nabla u(y))\cdot(x-y)}{|x-y|^2}}{|x-y|^{N+2\sigma}} \, dy
 \end{equation}
 We will use the following result whose proof is implicit in the proof of Lemma \ref{regional}

 \begin{lemma}
 Let $\Omega$ be a $\mathcal{C}^{1,1}$  domain such that its boundary, $\partial\Omega$, is a compact set and let $0<\alpha<1$. Then
 $$
 \int_{\mathbb R^N\setminus \Omega}\int_\Omega \frac 1{|x-y|^{N+\alpha}} dy\, dx<C_{N,\Omega}.
 $$
  \end{lemma}
 Using this and the fact that
 $$
 |\operatorname{div}\left((-\Delta)_{ \Omega}^\sigma (\nabla u)(x)\right)|\le C \int_\Omega\frac 1{|x-y|^{N+2\sigma}},
 $$
 we deduce our statement.

 \

 However,  when $1/2\le \sigma<1$ we do not have in general that
 $ \operatorname{div}\left((-\Delta)_{ \Omega}^\sigma (\nabla u)(x)\right)\in L^1(\mathbb R^N\setminus \Omega)$ as the following counterexample shows.

 \noindent{\bf Counterexample}: Let $\Omega$ denote the unit ball centered at the origin in $\mathbb R^N$. For $R$ large, define the function $u$ in the Schwartz class as follows
 $$u(x)= \begin{cases}
  \frac 12|x|^2, \quad &\mbox{ if } |x|\le R\\[3mm]
 0, \quad &\mbox{ if } |x|\ge 2R,
 \end{cases}
 $$
 and $u\in  \mathcal C^\infty$ everywhere.
 Then, formula \eqref{diver} gives for this $u$ and $1<|x|<R$,
 \begin{equation*}
 \operatorname{div}\left((-\Delta)_{ \Omega}^\sigma (\nabla u)(x)\right)=
 c_{N,\sigma}\int_\Omega \frac{-2\sigma}{|x-y|^{N+2\sigma}} \, dy.
 \end{equation*}
 This function is clearly not integrable in $B_R(0)\setminus \Omega$ for $2\sigma \ge 1$.

 Therefore the extra hypothesis in Proposition \ref{pro:intbyparts2} is necessary to justify our computations.

 \

 It is worth pointing out also that the integrability condition \eqref{condition3} is only needed in a local sense. More precisely, if $\Omega\subset B_R(0)$ then we always have for $u\in \mathcal S(\mathbb R^N)$ that $\operatorname{div}\left((-\Delta)_{ \Omega}^\sigma (\nabla u)(x)\right)\in L^1(\mathbb R^N\setminus B_{2R}(0))$.  In fact we have the following stronger estimate
 \begin{lemma}
 Assume as before that $\Omega\subset B_R(0)$. Then for every $u\in \mathcal S(\mathbb R^N)$ and every polynomial $p\in \mathcal P_1(\mathbb R^N)$ we have
 \begin{equation*} \int_{\mathbb{R}^N\setminus B_{2R}(0)}\left |\operatorname{div}\left((-\Delta)_{ \Omega}^\sigma \nabla u(x)\right)\,p(x) \right|\,dx<+\infty.
\end{equation*}
 \end{lemma}
 \begin{proof}
 To see this, we use the expression  given by \eqref{diver}. Since $|\Delta u(x) p(x)|<C$, $|x-y|\sim |x|$ for $|y|<R$
 and $|x|>2R$, and $\frac{|(\nabla u(x)-\nabla u(y))\cdot(x-y)\,p(x)|}{|x-y|^2}\le C
 \frac{|p(x)|}{|x|}\le C'$, for $|x|$ large, we have
 $$
 |\operatorname{div}\left((-\Delta)_{ \Omega}^\sigma \nabla u(x)\right)\,p(x)| \le C^{''} \frac 1{|x|^{N+2\sigma}}
 , \qquad |x|>2R.
 $$
 This finishes the proof.
 \end{proof}

 With all the above, we conclude that condition \eqref{eq:sagaghagdvjhgsjd} in Proposition \ref{pro:cuatrodiec} is always granted when $0<\sigma<1/2$ and is equivalent to condition \eqref{condition3} when $1/2\le \sigma <1$.
\section{The functional setting of the problem}\label{se:weakvaria}
We recall that a function $u$ is weakly differentiable in $ \mathbb R^N$ if there exists a vector field \\ $\vv U:\mathbb R^N \to \mathbb R^N$ such that
\begin{itemize}
\item $u, |\vv U| \in L^1_{loc}(\mathbb R^N)$ and
\item for every smooth vector field $\vv F$ of compact support  we have
$$
\int u(x)\, \mbox{div}\vv F(x) \, dx = -\int \vv U(x)\cdot \vv F(x) \, dx.
$$
\end{itemize}
We write $\vv U= \nabla u$. If $\vv U=(U^1,U^2,\dots,U^N)$, then the n'th component $U^j$ is denoted by $\partial_j u$ and satisfies
$$
\int \partial_j u\; \var \, dx=-\int u \;\partial_j \var\, dx, \quad \forall \var \in {\mathcal C}^\infty_0.
$$

We now define the appropriate  functional space  to solve the Neumann problem.
\begin{definition}
{Given $g_1$  as in the assumptions  $\mathcal A_{(f,g_1,g_2)}$, we define the space}
\begin{equation}\nonumber
H^s_{\mathcal N(g_1,g_2)}(\Omega)=\Big\{u:\mathbb R^N\rightarrow {\mathbb R}\,:\, u
{\mbox{ weakly differentiable }}\quad \text{and}\quad  \|u\|_{H^s_{\mathcal N(g_1,g_2)}{(\Omega)}}< +\infty\Big\},
\end{equation}
where
\begin{equation}\label{eq:canzperunamic}
\|u\|_{H^s_{\mathcal N(g_1,g_2)}(\Omega)}=\sqrt{\int_\Omega u^2\,dx+ \iint_{Q(\Omega)}\frac{|\nabla u(x)-\nabla u(y)|^2}{|x-y|^{N+2\sigma}}\, dx dy+ \int_{\mathcal C \Omega}|g_1|u^2\, dx}.
\end{equation}
\end{definition}
Notice that we have the formal function space identity
\begin{equation}\label{space1}
H^s_{\mathcal N(g_1,g_2)}(\Omega)=\dot{H}^s(\Omega)\bigcap L^2(\mathbb R^N,d\mu_{g_1}),
\end{equation}
with $\displaystyle \dot{H}^s(\Omega)$ and $d\mu_{g_1}$ defined in \eqref{spacebase} and \eqref{dmu} respectively.

\

\begin{remark}
Even {though} the space $H^s_{\mathcal N(g_1,g_2)}(\Omega)$ does not depend on the boundary data $g_2$, we prefer to include $g_2$ as a subscript in  the notation  {in order to keep in mind} both   Neumann conditions in problem \eqref{eq:p}.
\end{remark}

\

Let us prove the following

\begin{proposition}\label{Hilbert}
The space $H^s_{\mathcal N(g_1,g_2)}(\Omega)$ is a Hilbert space, with the inner product given by
$$
(u,v)_{H^s_{\mathcal N(g_1,g_2)}(\Omega)}=\int_{\mathbb R^N} uv\,d\mu_{g_1} +\iint_{Q(\Omega)}\frac{(\nabla u(x)-\nabla u(y))\cdot(\nabla v(x)-\nabla v(y))}{|x-y|^{N+2\sigma}}\, dx dy
.
$$
\end{proposition}

Clearly,  $$(\cdot,\cdot)_{H^s_{\mathcal N(g_1,g_2)}}:H^s_{\mathcal N(g_1,g_2)}(\Omega)\times H^s_{\mathcal N(g_1,g_2)}(\Omega)\rightarrow \mathbb R,$$
is a bilinear form defined over the reals. Moreover, if
$\|u\|_{H^s_{\mathcal N(g_1,g_2)}(\Omega)}=\sqrt{(u,u)_{H^s_{\mathcal N(g_1,g_2)}(\Omega)}}=0,$ we have on the one hand  that $D(u)=0$ and this says that $u$ coincides a.e. with a polynomial of degree 1. Since, on the other hand,
$\int_\Omega u^2 dx=0$ we conclude that $u$ must be 0 a.e. Hence, we only need to show that $H^s_{\mathcal N(g_1,g_2)}(\Omega)$ is complete.

\

Before proving that $H^s_{\mathcal N(g_1,g_2)}(\Omega)$ is complete we will state some technical results that will be needed. We will denote by
$\displaystyle \fint_A v$ the average integral value of $v$ on $A$, that is,
$\displaystyle \fint_A v=\frac 1{|A|}\int_A v.$

\begin{lemma}\label{all_2}
There exists a constant $C=C(N,|\Omega|)$ so that, for every given a ball $B$ with $\Omega \subset B$,  and $v:\mathbb R^N\rightarrow {\mathbb R}$ weakly differentiable with $|\nabla v|\in L^2(B)$, one has
\begin{equation}\label{Poincare1}
\int_B \left(v(x)-\fint_\Omega v\right)^2dx \le C\,|B|^{1+\frac 1N}\int_B |\nabla v(x)|^2dx.\end{equation}
\end{lemma}
\begin{corollary}\label{cor1} With the same hipotheses and notation of Lemma \ref{all_2}, we have
\begin{equation}\label{extension2}
\frac 12\fint_B |v(x)|^2dx \le C\,|B|^{1+\frac 1N}\fint_B |\nabla v(x)|^2dx+
\left(\left |\fint_\Omega v(y)\,dy\right|\right)^2.
\end{equation}
\end{corollary}
\begin{proof} Use simply the numerical inequality $(b-a)^2\ge \frac 12 b^2-a^2$.
\end{proof}

\

 \begin{proof}[Proof of Lemma \ref{all_2}]
The proof of  (\ref{Poincare1}) is  standard. First we observe that, from Jensen's inequality, we have
\begin{equation*}
\left(v(x)-\fint_\Omega v\right)^2=\left(\fint_\Omega (v(x)-v(y))dy\right)^2
\le \fint_\Omega \left(v(x)-v(y)\right)^2dy.
\end{equation*}
Integrating both sides with respect to $dx$ on $B$, and using the identity
\begin{equation*}
v(x)-v(y) =  \int_0^1 \nabla v(tx+(1-t)y)\cdot (x-y)dt, \quad \mbox{a.e.} \quad x,y \in \mathbb R^N,
\end{equation*}
and Jensen's again, we have
\begin{equation*}
\int_B \left(v(x)-\fint_\Omega v\right)^2dx \le \int_B \fint_\Omega\int_0^1\left |\nabla v(tx+(1-t)y)\cdot (x-y)\right |^2 dt\, dy\, dx.
\end{equation*}
By Fubini and the change of variables $x\to z=tx+(1-t)y$, we obtain that
\begin{eqnarray*}
J&:=& \int_B \fint_\Omega\int_0^1\left |\nabla v(tx+(1-t)y)\cdot (x-y)\right |^2 dt\, dy\, dx\\[2mm]
&\le & \fint_\Omega  \int_0^1 \int_{B} \left|\nabla v (z)\right|^2\left( \frac{|z-y|}{t}\right)^2 \chi_B\left( \frac{z-y}{t}+y\right)dz\frac{dt}{t^N} dy,
\end{eqnarray*}
where we have used that $\Omega\subseteq B$. We observe now that if the ball $B$ has radius $R$ and both, ${y}$ and $\frac{z-y}{t}+y $ are in $B$, then
$\left|\frac{z-y}{t}\right|< 2R$, which forces $t$ to be bigger than $\frac{|z-y|}{2R}$. Thus,
\begin{eqnarray*}
J&\le& \int_{B} \left| \nabla v(z)\right|^2 \fint_\Omega  \int_{1\wedge \frac{|z-y|}{2R}}^1 |z-y|^2\frac{dt}{t^{N+2}} dy\, dz \\[2mm]
&\le&\frac {(2R)^{N+1}}{N+1} \int_{B} \left|\nabla v(z)\right|^2 \fint_\Omega \frac{1}{|z-y|^{N-1} }dy\, dz,
\end{eqnarray*}
Finally, using that $\displaystyle \int_\Omega \frac{1}{|z-y|^{N-1} }dy \le
C(N)|\Omega|^{1/N}$, we conclude the lemma.
\end{proof}

\

 Now we  prove the following

\begin{lemma}\label{Wintinger} If $u\in \dot{H}^s(\Omega)$ then $|\nabla u|\in L^2(B)$ for every ball $B$. Moreover, if $ \Omega\subset B$ one has the estimate  \begin{equation}\label{Wintinger2}
\int_B\left|\nabla u(x)-\fint_\Omega \nabla u\right|^2 dx\le {C(N,\sigma)}\,
|B|^{1+\frac {2\sigma}N} \int_B\fint_\Omega\frac{|\nabla u(x)-\nabla u(y)|^2}{|x-y|^{N+2\sigma}}dxdy \end{equation}
\end{lemma}

\

As an easy consequence we obtain the following inequality
\begin{corollary}\label{Wintinger0} There exists a positive constant $C=C(N)>0$ such that for  $u\in \dot{H}^s(\Omega)$ and every ball $B\supset \Omega$,
$$
\frac 12\int_B|\nabla u(x)|^2 dx\leq \frac{{C(N,\sigma)}}{|\Omega|}\,
|B|^{1+\frac {2\sigma}N}\, D^2(u)
+ \left(\fint_\Omega \nabla u(y) dy\right)^2,
$$
where $D(u)$ was given in \eqref{spacebase}. In particular, if $u\in \dot{H}^s(\Omega)$ then $u \in H^1_{loc}(\mathbb R^N)$
\end{corollary}

\begin{proof}[Proof of Lemma \ref{Wintinger}] To simplify the notation, set
$$\displaystyle \vv {b}=(b^1,\ldots,b^N):=
\left(\fint_\Omega \partial_1 u,\dots, \fint_\Omega \partial_N u\right)=
\displaystyle \fint_\Omega  \nabla u(y) dy.$$
As in the proof of Lemma \ref{all_2}, we easily get
\begin{equation*}
\int_B\left|\nabla u(x)-\vv b\right|^2\, dx
\leq \int_{B}\fint_{\Omega}\left|\nabla u(x)-\nabla u(y)\right|^2\, dy\, dx.\nonumber
\end{equation*}
Therefore, if $B$ is a ball that contains $\Omega$ and has radius $R$, we obtain
$$\int_B\left |\nabla u(x)dx-\vv {b}\right|^2 dx \le(2R)^{N+2\sigma}
\int_B\fint_\Omega \frac{|\nabla u(x)-\nabla u(y)|^2}{|x-y|^{N+2\sigma}}dxdy, $$
as stated.
\end{proof}

\

\begin{proof}[Proof of Proposition \ref{Hilbert}]
As we pointed out above, we only need to show that $H^s_{\mathcal N(g_1,g_2)}(\Omega)$ is complete.
 To that end, consider a Cauchy sequence $\{u_k\}_k$ in our space. We proceed in several steps:

\

\textsc{Step 1:} There exists a function $u^*$ such that
\begin{equation}\label{L2mu}
\lim_{k\to\infty} \left(\int_{\Omega} |u^*-u_k|^2\, dx + \int_{\mathcal C \Omega} |u^*-u_k|^2 |g_1| \, dx
\right)=0.
\end{equation}
This comes simply from the fact that $\{u_k\}_k$ is a Cauchy sequence in
$ L^2(\mathbb{R}^N, d\mu_{g_1})$. Since, in particular,
\begin{equation*}
\lim_{k\to\infty} \int_{\Omega\bigcup \{|g_1|>1/m\}} |u^*-u_k|^2\, dx =0, \end{equation*}
for all $m\in \mathbb N$, there exists a subsequence that converges pointwise to $u^*$ in the set $\Omega\bigcup \{g_1\neq 0\}$ a.e. with respect to Lebesgue measure.

\

\textsc{ Step 2:} There exists a vector field $\vv U:\mathbb R^N \to \mathbb R^N$ such that for every ball $B\subset \mathbb R^N$ we have
\begin{equation}\label{nabla:L2}
\lim_{k\to\infty} \int_B |\nabla u_k-\vv U|^2\, dx=0.
\end{equation}
The idea here is to prove that the sequence of vector fields $\{\nabla u_k\}_k$ is a Cauchy sequence in  $[L^2(B)]^N$. By using \eqref{Wintinger2}
and putting as above
\[\vv {b_k}=(b_k^{1},\ldots,b_k^{N}):= \displaystyle \fint_\Omega  \nabla u_k(y) dy,\]
we find that the sequence of vector fields $\{\nabla u_k-\vv{b_k}\}_k$ is a Cauchy sequence in  $[L^2(B)]^N$ for every ball $B\subset \mathbb R^N$ and, hence, there exists a vector field $\vv U_0=(U_0^1,\ldots,U_0^N)$ so that
\begin{equation}\label{nabla0:L2}
\lim_{k\to\infty} \int_B |\nabla u_k(x)-\vv{b_k}-\vv U_0(x)|^2\, dx=0, \quad \forall
 B \mbox{ ball} .
\end{equation}

\

Let us prove that the sequence of vectors $\{\vv{b_k}\}_k$ has a limit. To see it, we observe that if $\varphi$ is a smooth bump function supported in $\Omega$ with $\int \varphi \, dx=1$ we have
\begin{equation}\label{vector}
\lim_{k\to\infty} \int \left(\partial_j u_k-b_k^j\right)\varphi = \int U_0^j\varphi,
  \quad  j=1,\dots, N.
\end{equation}
Since $\displaystyle \int \left(\partial_j u_k-b_k^j\right)\varphi \, dx=
-\int u_k\,\partial_j \varphi\, dx -b_k^j $ and $u_k\rightarrow u^*$ as $k\to \infty$
in $L^2(\Omega, dx)$, we have that there exists the limit
$$
b_0^j:=\lim_{k\to\infty} b_k^j=-\int\left(u^*\,\partial_j \varphi+U_0^j\varphi\right)dx.
$$
If we set $\vv b_0=(b_0^1,\ldots,b_0^N)$ then $\vv U=\vv U_0+\vv b_0$ represents the vector field seeked in \ref{nabla:L2}.

\

\textsc{Step 3:} From Corollaries \ref{cor1} and \ref{Wintinger0} we have that the family $\{u_k\}_k$ is a Cauchy sequence on $L^2(B,dx)$ for every ball $B\subset \mathbb R^N$. In particular, there exists a function $u$ defined on all $\mathbb R^N$ so that
\begin{equation}\label{L1}
u_k\longrightarrow u \quad \mbox{in } {L^2_{loc}(\mathbb R^N)}, \quad \mbox{as } k\to\infty.
\end{equation}
Since, from \textsc{Step 2}, we also have
\begin{equation}\label{L1_grad}
\nabla u_k\longrightarrow \vv U \quad \mbox{in } L^2_{loc}(\mathbb R^N), \quad \mbox{ as } k\to\infty,
\end{equation}
we conclude that $\vv U=\nabla u$.
Obviously we also have that $u=u^*$ a.e. on the set $\{x\in\mathbb{R}^{N}:g_1(x)\neq 0\}$.

\

We collect now all the information to prove that the function $u$ is indeed the limit of the sequence $\{u_k\}_k$ in the norm of $H^s_{\mathcal N(g_1,g_2)}(\Omega)$.  First, we have from (\ref{L1_grad}) and Fatou's Lemma that
$$
\lim_{k\to\infty} D^2(u-u_k)=\lim_{k\to\infty} \iint_{Q(\Omega)}\frac{\left|(\vv U(x)-\nabla u_k(x))-(\vv U(y)-\nabla u_k(y))\right|^2}{|x-y|^{N+2\sigma}}\,
dx dy =0.
$$
This, together with (\ref{L2mu}) and the above observation on $u^*$ gives
\begin{eqnarray*}
\lim_{k\to\infty} \|u-u_k\|^2_{H^s_{\mathcal N(g_1,g_2)}} &=&
 \lim_{k\to\infty} \left(\int |u-u_k|d\mu_{g_1}+D^2(u-u_k)\right) \\
 &=& \lim_{k\to\infty} \left(\int |u^*-u_k|d\mu_{g_1}+D^2(u-u_k)\right)=0.
\end{eqnarray*}
This finishes the proof of Proposition \ref{Hilbert}
\end{proof}
\section{Existence of solutions to \eqref{eq:p}. The proof of Theorem \ref{main}.} \label{sec:variational}
We start defining  the following weak formulation for the problem \eqref{eq:p}. We have
\begin{definition}\label{def:weaksolution}
Assume that $f\in L^2(\Omega)$, $g_1\in {L^1( \mathbb{R}^N\setminus\Omega, |x|^2\, dx )}{\cap L^1(\mathbb{R}^N\setminus\Omega)},$
and $g_2\in L^2(\partial\Omega)$.  Then  $u\in H^s_{\mathcal N(g_1,g_2)}(\Omega)$ is a weak solution to \eqref{eq:p} if and only if
\begin{eqnarray}\label{eq:dbahgduvasc}
&&\frac{c_{N,\sigma}}{2} \int_{Q(\Omega)}\frac{(\nabla u(x)-\nabla u(y))\cdot(\nabla v(x)-\nabla v(y))}{|x-y|^{N+2\sigma}}\, dx dy
\\\nonumber
&&=\int_{\Omega}fv\,dx+\int_{\mathbb R^N\setminus \Omega}g_1v\, dx+ \int_{\partial \Omega}g_2v\, dS,
\end{eqnarray}
for all $v \in H^s_{\mathcal N(g_1,g_2)}(\Omega)$.
\end{definition}
\begin{remark}\label{rem:wellposdness}
{We point out that if $u,v \in H^s_{\mathcal N(g_1,g_2)}(\Omega)$, each term in \eqref{eq:dbahgduvasc} is well defined. In particular since $v\in H^1(\Omega)$,  we deduce that  $v$ has a trace $Tv$ on $\partial \Omega$,  in $L^2(\partial \Omega)$. The regularity of $g_2$ can also  be sharpened according to the trace theory, that is, it is sufficient to require that $g_2\in L^{q}(\partial\Omega)$ whit $q=2(N-1)/N<2$ (see \cite{dibenedetto}).
Moreover, by Sobolev inequality,  see  \cite{DnPV}, {since ${v} \in L^p(\Omega)$, with  $p\leq2N/(N-2(\sigma+1))$ (we make the convention that $2N/(N-2(\sigma+1)=+\infty$ if  $N\leq 2(\sigma+1)$), the previous definition has sense for every} $f \in L^q(\Omega)$, with  $q \geq 2N/(N+2(\sigma+1))$.}
\end{remark}
Thanks to Definition \ref{def:weaksolution}  we can also associate a variational formulation to \eqref{eq:p}. If $f\in L^2(\Omega)$, $g_1\in {L^1( \mathbb{R}^N\setminus\Omega, |x|^2\, dx )}{\cap L^1(\mathbb{R}^N\setminus\Omega)},$
and $g_2\in L^2(\partial\Omega)$, for all $u\in H^s_{\mathcal N(g_1,g_2)}(\Omega)$ we can define the functional
\begin{equation}\label{eq:deffunctional}
J(u):= \frac{c_{N,\sigma}}{4} \int_{Q(\Omega)}\frac{|\nabla u(x)-\nabla u(y)|^2}{|x-y|^{N+2\sigma}}\, dx dy-\int_{\Omega}f{u}\,dx-\int_{\mathbb R^N\setminus \Omega}g_1{u}\, dx- \int_{\partial \Omega}g_2{u}\, dS,
\end{equation}

\

If for example, we consider the homogeneous problem
\begin{equation}\label{nonresonant}
\begin{cases}
u+(-\Delta)^s  u =  f(x) &\text{in}\,\,\Omega\\
\mathcal N_\sigma^1  u  =0&\text{in}\,\,\mathbb{R}^N\setminus\overline\Omega\\
\mathcal N_\sigma^2 u=0&\text{on}\,\,\partial\Omega,
\end{cases}
\end{equation}
it is easy to see that a standard variational argumentation gives the unique energy solution.

\

Along this section we analyze a compatibility condition to take into account,  to prove the existence of weak solutions of \eqref{eq:p}, that is,  in the resonant case.
A key point is the following: let us  consider in $H^s_{\mathcal N(g_1,g_2)}(\Omega)$ the equivalence relation defined by
\[u\backsim v \quad \hbox{ if an only if  there exists } p\in \mathcal{P}_1(\mathbb{R}^N) \hbox{ such that  } \quad u=v+p.\]
{Let us denote by} $\mathcal{H}^s$ the quotient space {with respect to} this equivalence relation, {that is}
$${\mathcal{H}^{s}:={H^s_{\mathcal N(g_1,g_2)}(\Omega)}\diagup{\mathcal{P}_{1}(\mathbb{R}^{N})}=\Big\{ [u],\, u\in H^s_{\mathcal N(g_1,g_2)}(\Omega)\Big\},}$$
{where, given $u\in H^s_{\mathcal N(g_1,g_2)}(\Omega)$
$$[u]=\{v\in H^s_{\mathcal N(g_1,g_2)}(\Omega):\, v\backsim u\}:=\{u+p:\, u\in H^s_{\mathcal N(g_1,g_2)}(\Omega),\, p\in\mathcal{P}_{1}(\mathbb{R}^{N})\}\subseteq H^s_{\mathcal N(g_1,g_2)}(\Omega).$$}
It is well known that
$$\|[u]\|^2_{\mathcal{H}^{s}}=\inf_{p\in\mathcal{P}_1(\mathbb{R}^{N})}\|u-p\|^{2}_{H^s_{\mathcal N(g_1,g_2)}(\Omega)}.$$
By the Hilbert projection theorem it is clear that the previous infimum is attained, that is there exists $\widetilde{p}\in\mathcal{P}_{1}(\mathbb{R}^{N})$ such that
$$\|[u]\|^2_{\mathcal{H}^{s}}=\|u-\widetilde{p}\|^{2}_{H^s_{\mathcal N(g_1,g_2)}(\Omega)}.$$
Moreover $v:=u-\widetilde{p}\, {\in}\,  H^{s,0}_{\mathcal{N}(g_1,g_2)}(\Omega)$ where
\begin{equation}\label{noviembre_set}
H^{s,0}_{\mathcal{N}(g_1,g_2)}(\Omega):=\left\{u\in H^s_{\mathcal{N}(g_1,g_2)}(\Omega):\, \int_{\mathbb{R}^{N}}{u p\, d\mu_{g_1}}=0,\, {\forall} p\in\mathcal{P}_{1}(\mathbb{R}^{N})\right\}
\end{equation}
where $d\mu_{g_1}$ was defined in \eqref{dmu}. We notice that $H^{s,0}_{\mathcal{N}(g_1,g_2)}(\Omega)=H_{g_1}^{s,0}(\Omega)\cap L^{2}(\mathbb{R}^{N}, d\mu_{g_1})$ is a closed subspace of $H^s_{\mathcal N(g_1,g_2)}(\Omega)$.
Let us define
\[\|[u]\|^2_*=\int_{Q(\Omega)}\frac{|\nabla u(x)-\nabla u(y)|^2}{|x-y|^{N+2\sigma}}\, dx dy,\]
that is a norm in $\mathcal{H}^{s}$. In fact for $p\in\mathcal{P}_{1}(\mathbb{R}^{N})$,  we have that $\|u+p\|_*=0$ implies that $u\in \mathcal{P}_{1}(\mathbb{R}^{N})$, that is the zero function in $\mathcal{H}^{s}$.

We prove the following result that shows an example of an admissible $g_1$.

\begin{lemma}\label{equivalencia}
Let us suppose   $g_1\in L^p_{c}(\mathbb R^N\setminus \Omega),$
with $p>N/2$, then   $\lambda_1(g_1)>0$,  where $\l_1(g_1)$ is defined in \eqref{espectral}. As a consequence, the norm $\|\cdot\|_{\mathcal{H}^{s}}$ is equivalent to the norm $\|\cdot\|_*$, that is there exists a positive constant $C=C(N,\sigma, \Omega)$ such that
\begin{equation}\label{tutoria}
\frac1C\|[u]\|_*\leq\|[u]\|_{\mathcal{H}^{s}}\leq C \|[u]\|_*, \qquad \text{for all}\,\, [u]\in \mathcal{H}^{s}.
\end{equation}
\end{lemma}
\begin{proof}
{The fact that} $g_1$ verifies that $g_1\in {L^1( \mathbb{R}^N\setminus\Omega, |x|^2\, dx )}{\cap L^1(\mathbb{R}^N\setminus\Omega)}$ {easily follows} by a simple application of the H\"{o}lder inequality. {Moreover since $g_1$ has compact support, by  Corollary~\ref{Wintinger0},  we deduce that $H^{s,0}_{g_1}(\Omega) = H^{s,0}_{\mathcal{N}(g_1,g_2)}(\Omega).$}
{Show that $\lambda_1(g_1)>0$ it is equivalent to obtain, for every  $v\in   H^{s,0}_{g_1}(\Omega)= H^{s,0}_{\mathcal{N}(g_1,g_2)}(\Omega)$, the following Poincar\'e-type inequality}
\begin{equation}\label{poincare}
\int_{\Omega}{v}^2\,dx+ {\int_{\mathcal C \Omega}|g_1|{v}^2\, dx}
\leq C(N,\sigma, \Omega) \iint_{Q(\Omega)}\frac{|\nabla {v}(x)-\nabla {v}(y)|^2}{|x-y|^{N+2\sigma}}\, dx dy.
\end{equation}
Observe that if \eqref{poincare} is true, then the  second inequality of \eqref{tutoria} is also valid. Indeed if we consider $[u]\in\mathcal{H}^s$ and
\begin{equation}\label{eq:wwww}
w:=u-\widetilde{p}\in H^{s,0}_{\mathcal{N}(g_1,g_2)}(\Omega),\quad  \widetilde{p}\in\mathcal{P}_{1}(\mathbb{R}^{N}),
\end{equation}
the function where the infimum in the norm is attained, by \eqref{poincare} it will follow that
\begin{eqnarray*}
\|[u]\|^2_{\mathcal{H}^{s}}=\|w\|^2_{H^{s}_{\mathcal{N}(g_1,g_2)}(\Omega)}&=& \int_{\Omega}w^2\,dx+\iint_{Q(\Omega)}\frac{|\nabla w(x)-\nabla w(y)|^2}{|x-y|^{N+2\sigma}}\, dx dy+ {\int_{\mathcal C \Omega}|g_1|w^2\, dx}
\\\nonumber
&\leq& C(N,\sigma, \Omega) \iint_{Q(\Omega)}\frac{|\nabla u(x)-\nabla u(y)|^2}{|x-y|^{N+2\sigma}}\, dx dy\leq C(N,\sigma, \Omega) \|[u]\|_*.
\end{eqnarray*}
as wanted.

 To show \eqref{poincare} let us suppose, by contradiction, that there exists, up to a renormalization, a sequence $\{v_k\}\subset  H^{s,0}_{\mathcal{N}(g_1,g_2)}(\Omega)$ such that
\begin{equation}\label{ppl}
\int_{\Omega}v_k^2\, dx+\int_{\mathbb{R}^N\setminus \Omega}|g_1|v_k^2\, dx=1\,\, \mbox{and}\,\, \iint_{Q(\Omega)}\frac{|\nabla v_k(x)-\nabla v_k(y)|^{2}}{|x-y|^{N+2\sigma}}\, dx dy<\frac{1}{k}.
\end{equation}
First of all, we will show that actually
\begin{equation}\label{orden_noviembre}
\int_{\Omega}v_k^2\, dx+\int_{\Omega}|\nabla v_k|^2\, dx+\int_{\Omega}\int_{\Omega}\frac{|\nabla v_k(x)-\nabla v_k(y)|^2}{|x-y|^{N+2\sigma}}\, dx dy:=\|v_k\|^2_{W^{s,2}(\Omega)}<C.\end{equation}
In fact, by contradiction, let us suppose that there exists a subsequence that we still  denote by $\{v_k\}$, such that
\begin{equation}\label{pio}
\rho_k:=\int_{\Omega}|\nabla v_k|^2\, dx\rightarrow +\infty.\
\end{equation}
Defining $z_k=v_k/\rho_k$, since $\|\nabla z_k\|_{L^2(\Omega)}=1$, from \eqref{ppl} is clear that
\begin{equation}\label{wui}
\int_{\Omega}z_k^2\, dx+\int_{\Omega}|\nabla z_k|^2\, dx+\int_{\Omega}\int_{\Omega}\frac{|\nabla z_k(x)-\nabla z_k(y)|^2}{|x-y|^{N+2\sigma}}\, dx dy< C,
\end{equation}
that is, $\|z_k\|^2_{W^{s,2}(\Omega)}\leq C$ so  $z_k \rightharpoonup z^\star$ in $W^{s,2}(\Omega)$. In particular, by Rellich's theorem we have that, up to subsequence
$z_k\overset{L^2(\Omega)}{\longrightarrow} z^\star.$
Moreover by \eqref{ppl} and \eqref{pio}-\eqref{wui} it follows that
\begin{eqnarray*}
C&\geq& \int_{\Omega}z_k^2\, dx+\int_{\Omega}\int_{\Omega}\frac{|\nabla z_k(x)-\nabla z_k(y)|^2}{|x-y|^{N+2\sigma}}\, dx dy\\
&=&\frac{1}{\rho_k^2}\left(\int_{\Omega}v_k^2\, dx+\int_{\Omega}\int_{\Omega}\frac{|\nabla v_k(x)-\nabla v_k(y)|^2}{|x-y|^{N+2\sigma}}\, dx dy\right)\to 0,\, k\to\infty.
\end{eqnarray*}
So that in particular $z^*\equiv0$. On the other hand using the fractional compact embeddings theorem \cite[Theorem 7.1]{DnPV},
we have that (up to subsequence)
$$1=\lim_{k\rightarrow +\infty} \int_{\Omega}|\nabla z_k|^2\, dx= \int_{\Omega}|\nabla z^*|^2\, dx,$$
which is a contradiction and therefore \eqref{orden_noviembre} follows. Thus, from \eqref{orden_noviembre} in particular we infer that $\{v_k\}$ is bounded in $W^{s,2}(\Omega)$, so, up to subsequence, $v_k\rightharpoonup v$ in $W^{s,2}(\Omega)$ and $v_k\to v$ in $L^{2}(\Omega)$. Moreover by Lemma~\ref{Wintinger},  Lemma \ref{all_2} and from that fact that $\|\nabla v_k\|_{L^2(\Omega)}\leq C$ we get that
\[\int_{B}v_k^2\,dx+\int_{B}|\nabla v_k|^2\,dx\leq C,\]
where $B$ is a ball centered at the origin with $\Omega \subset B$ and $C=C(N,|B|,|\Omega|)$ is a positive constant. That is, $\|v_k\|^2_{H^1(B)}\leq C$, so that $v_k\to v$ in $L^{q}$, $q<2N/(N-2)$. Hence, using the fact that $g_1\in L^p_c(\mathbb R^N\setminus\Omega),\, p>N/2$ we can pass to the limit in \eqref{ppl} getting that
\begin{equation}\label{eq:miaoiooio}
\int_{\Omega}v^2\, dx+\int_{\mathbb{R}^{N}\setminus \Omega}|g_1|v^2\, dx=1.
\end{equation}
By  the lower semicontinuity of the norm w.r.t. the weak convergence, form \eqref{ppl} is also clear that
$$\iint_{Q(\Omega)}\frac{|\nabla v(x)-\nabla v(y)|^{2}}{|x-y|^{N+2\sigma}}\, dx dy=0.$$
So that $v\in \mathcal{P}_{1}(\mathbb{R}^N)$ which, by  \eqref{eq:miaoiooio}, clearly implies a contradiction with the fact that $v \in H^{s,0}_{\mathcal{N}(g_1,g_2)}(\Omega)$.

\

{To conclude the proof of Lemma let us mention that the first inequality of \eqref{tutoria} is obviously true because $\|[u]\|^2_{\mathcal{H}^{s}}=\|w\|^2_{H^{s}_{\mathcal{N}(g_1,g_2)}(\Omega)}$ where $w$ was given in \eqref{eq:wwww}.}
\end{proof}
{Next we will emphasize that $J$ is well defined in  $\mathcal{H}^s$. In fact if $f, g_1$ and $ g_2$ satisfies de compatibily condition \eqref{eq:unodiez}  and $u \backsim v$ then
\begin{equation}\label{jota}
J(u)=J(v)=J(u-p).
\end{equation}
Therefore we can establish now the following}
\begin{theorem}\label{key_th}
Assume that $(\mathcal A_{(f,g_1,g_2)})$ holds and let $J:\mathcal{H}^s  \rightarrow \mathbb R$ be the functional defined in \eqref{eq:deffunctional}. {If $f, g_1$ and $ g_2$ satisfying the compatibility condition \eqref{eq:unodiez} then}
\begin{enumerate}
\item $J$ has a unique minimum in $\mathcal{H}^s$.
\item Every critical point of $J$ is in fact a weak solution to the problem \eqref{eq:p} modulo a polynomial in $\mathcal{P}_1(\mathbb{R}^N)$.
\end{enumerate}
\end{theorem}
\begin{proof}
First of all, it is easy to check (see also Remark  \ref{rem:wellposdness}) that the functional $J(u)$ is well defined in $\mathcal{H}^s$ that is, it is enough to prove that
\begin{equation}\label{vv}
\big |J([u])\big |< \infty.
\end{equation}
By abuse of notation, taking into account \eqref{jota} we will write $J(u)$ instead of $J({[u]})$. To obtain \eqref{vv} it is sufficient to point out that we have
\[\left|\int_{\Omega}fu\, dx\right|\leq \|f\|_{L^2(\Omega)}\|u\|_{L^2(\Omega)} \leq C\|[u]\|_{\mathcal{H}^s}.\]
Moreover by the Cauchy-Schwartz inequality,
$$\left|\int_{\mathbb{R}^N\setminus\Omega} ug_1 dx\right|\leq \int_{\mathbb{R}^N\setminus\Omega}|g_1|^{\frac 12}|g_1|^{\frac 12}|u|\,dx\leq
\left(\int_{\mathbb{R}^N\setminus\Omega}|g_1|\, dx\right)^{\frac 12}\left(\int_{\mathbb{R}^N\setminus\Omega}|g_1||u|^2\, dx\right)^{\frac 12}\leq C\|[u]\|_{\mathcal{H}^s}.$$
On the other hand, by using the H\"older and trace inequality and the Poincar\'e-Wintinger inequality given in Corollary \ref{Wintinger0},  we get that
\[\left|\int_{\partial\Omega}  g_2 u\, dS\right|\le \|g_2\|_{L^{\frac{2(N-1)}{N}}(\partial\Omega)}\|u\|_{L^{\frac{2^*(N-1)}{N}}(\partial\Omega)} \leq C\|g_2\|_{L^{2}(\partial\Omega)}(\|u\|_{L^2(\Omega)}+\|\nabla u\|_{L^2(\Omega)})\leq C(1+\|[u]\|_{\mathcal{H}^s}).\]
By the previous computations and the fact that $\lambda_1(g_1)>0$, {(see \eqref{tutoria}-\eqref{eq:wwww})} we also deduce that  $J$ is  coercive in $\mathcal{H}^s$, that is
\[J(u)\geq C_1\|{[u]}\|^2_{\mathcal{H}^s}-C_2\|{[u]}\|_{\mathcal{H}^s}-C_3,\]
for some positive constants $C_1, C_2, C_3$. As $J$ is continuous, convex and coercive  then, $(1)$ is an elementary consequence  of the classical minimization results.

To obtain (2) let us consider $[u],[v] \in \mathcal{H}^s$ and $t\in \mathbb R$. We   deduce that
\begin{eqnarray}\label{eq:tresonce}
&&\lim_{t\rightarrow 0}\frac{J(u+tv)-J(u)}{t}
\\\nonumber&&=
\frac{c_{N,\sigma}}{2} \int_{Q(\Omega)}\frac{(\nabla u(x)-\nabla u(y))\cdot(\nabla v(x)-\nabla v(y))}{|x-y|^{N+2\sigma}}\, dx dy-\int_{\Omega}fv \, dx
\\\nonumber
&&-\int_{\mathbb R^N\setminus \Omega}g_1v\, dx- \int_{\partial \Omega}g_2v\, dS.
\end{eqnarray}
In fact to get \eqref{eq:tresonce}, we observe that the first term on the r.h.s. of \eqref{eq:deffunctional} can be view as a bilinear form and the other terms are linear. From  \eqref{eq:tresonce}  we obtain the conclusion, that is
\begin{eqnarray}\label{eq:congrrssjajoj}
J'(u)[v]&=&
\frac{c_{N,\sigma}}{2} \int_{Q(\Omega)}\frac{(\nabla u(x)-\nabla u(y))\cdot(\nabla v(x)-\nabla v(y))}{|x-y|^{N+2\sigma}}\, dx dy-\int_{\Omega}fv \, dx
\\\nonumber
&-&\int_{\mathbb R^N\setminus \Omega}g_1v\, dx- \int_{\partial \Omega}g_2v\, dS,
\end{eqnarray}
for all $[u],[v] \in \mathcal{H}^s$ and therefore, a critical point of $J$ is in fact a weak  solution (in the sense of Definition \ref{def:weaksolution}) to~\eqref{eq:p} modulo first degree polynomials.
\end{proof}
We next show a lemma useful to obtain the proof of Theorem \ref{main} because  show that the compatibility condition is a necessary condition for the existence of a solution to~\eqref{eq:p}:
\begin{lemma}[Necessary condition]\label{pro:necesscondit}
Let us suppose that $(\mathcal A_{(f,g_1,g_2)})$ hold  and let $u$ be  a weak solution to~\eqref{eq:p}. Then {\eqref{eq:unodiez} is satisfied. That is,}
$$\int_{\Omega} f p \, dx +\int_{\mathbb R^N\setminus \Omega}g_1p \, dx+\int_{\partial \Omega}g_2 p\, dS=0, \qquad \text{for all }\,\,  p\in \mathcal P_1(\mathbb R^N).$$
\end{lemma}
\begin{proof} It is sufficient to observe that   $\mathcal P(x)\subset H^s_{\mathcal N(g_1,g_2)}(\Omega)$. Therefore using  $p\in \mathcal P(x)$ as a  test function in \eqref{eq:dbahgduvasc} and taking into account that $\nabla p(x)$ is a constant function we conclude.
\end{proof}
\begin{proof}[Proof of Theorem \ref{main}:]

By Lemma~\ref{pro:necesscondit} it is clear that if there exists  a  weak solution $u\in H^s_{\mathcal N(g_1,g_2)}(\Omega)$ to~\eqref{eq:p}, then \eqref{eq:unodiez} is obtained. On the contrary if \eqref{eq:unodiez} is true, then by Theorem \ref{key_th} there exists  $[u]\in \mathcal{H}^{s}$ a solution of \eqref{eq:p}. The solution is unique up to a polynomial $p\in \mathcal P (\mathbb R^N).$
\end{proof}
\
 The next lemma  will be useful in order to prove the right uniqueness result for weak solutions to \eqref{eq:p} and to analyze the spectral properties of the Neumann Problem (see Section 5). We notice here that this result is the equivalent of \cite[Lemma 3.8]{DRoV} for the Neumann problem associated to the fractional Laplacian operator of order $0<s<1$.
\begin{lemma}\label{eq:uniquuuueneessssss}
Let assume that $(\mathcal A_{(f,g_1,g_2)})$ hold  and let $u$ be  a weak solution to
\begin{equation}\nonumber
\begin{cases}
(-\Delta)^s  u =  f(x) &\text{in}\,\,\Omega\\
\mathcal N_\sigma^1  u  =g_1&\text{in}\,\,\mathbb{R}^N\setminus\overline\Omega\\
\mathcal N_\sigma^2 u=g_2&\text{on}\,\,\partial\Omega,
\end{cases}
\end{equation}
with $f,g_1,g_2$ non negative functions. Then
\[u\in \mathcal P_1(\mathbb R^N).\]
\end{lemma}
\begin{proof}
Taking $\mathcal P_1(\mathbb R^N)\ni p\equiv 1$ {as a test function} we get
\begin{equation}\nonumber
\int_{\Omega} f  \, dx +\int_{\mathbb R^N\setminus \Omega}g_1 \, dx+\int_{\partial \Omega}g_2 \, dS=0,
\end{equation}
and thus, since $f,g_1,g_2$ are non negative, we deduce that $f=0$ a.e. in $\Omega$, $g_1=0$ a.e. in $\mathbb R^N\setminus\Omega$ and $g_2=0$ a.e. (w.r.t the measure $S$ of the boundary) on $\partial \Omega$. Therefore considering {now} $v=u$ as a test function we get
\[\int_{Q(\Omega)}\frac{|\nabla u(x)-\nabla u(y)|^2}{|x-y|^{N+2\sigma}}\, dx dy=0,\]
that is $\nabla u (x)$ is constant in $\mathbb R^N$. Thus  $u\in \mathcal P_1(\mathbb R^N).$
\end{proof}

\

Next we will analyze the existence of the resonant problem with a different approach that, in particular, will be useful to study the spectrum of the Neumann problem $\eqref{eq:p}$ in the next section. This is the approach done in \cite{DRoV} for $0<s<1$.

We start by considering the  problem \eqref{nonresonant} with homogeneous Neumann condition, namely we set   $g_1=0$ in $\mathbb R^N\setminus \overline \Omega$ and $g_2=0$ on $\partial \Omega$. We also assume that $f \not \equiv 0$, since otherwise the result holds considering the trivial solution.

We call, to be short,  $H^s_{\mathcal N, \vec 0}(\Omega)$,  the space $H^s_{\mathcal N(g_1,g_2)}(\Omega)$ with homogeneous Neumann conditions $g_1=g_2=0$ in the problem \eqref{eq:p}.

First of all we observe that, by the Riesz theorem, given $h\in L^2(\Omega)$, since the functional
$$
v\longrightarrow \int_{\Omega}hv\,dx,\quad v\in {H^s_{\mathcal N, \vec 0}(\Omega)}
$$
is linear and continuous in $ {H^s_{\mathcal N, \vec 0}(\Omega)},$ there exists a unique function $w\in H^s_{\mathcal N, \vec 0}(\Omega)$ such that
\begin{equation}\label{eq:ausiliarioproblema}
\int_{\Omega}wv\, dx+\frac{c_{N,\sigma}}{2} \int_{Q(\Omega)}\frac{(\nabla w(x)-\nabla w(y))\cdot(\nabla v(x)-\nabla v(y))}{|x-y|^{N+2\sigma}}\, dx dy
=\int_{\Omega}hv\,dx,
\end{equation}
for all $v\in {H^s_{\mathcal N, \vec 0}(\Omega)}$, with  $\mathcal N^1_\sigma w(x)=0$ in $\mathbb R^N\setminus \overline \Omega$ and  $\mathcal N^2_\sigma w(x)=0$ on $\partial \Omega$.
Therefore we will define  the inverse operator
\begin{eqnarray}\nonumber
K:L^2(\Omega)&\longrightarrow& {H^s_{\mathcal N, \vec 0}(\Omega)}\\\nonumber
h&\longrightarrow& w,
\end{eqnarray}
with $w$ the solution to \eqref{eq:ausiliarioproblema}. We can also define the restriction  operator $\overset{\circ}{K}$ as
\begin{equation}\label{eq:comaaaavasc}
\overset{\circ}{K}h=Kh\big|_{\Omega},
\end{equation}
and readily follows that $\overset{\circ}{K}:L^2(\Omega)\longrightarrow H^s_{\mathcal N, \vec 0}(\Omega)\subseteq L^2(\Omega)$.

Notice that we can use the Fredholm alternative, given that $\overset{\circ}{K}$ is compact. Indeed, taking  $w$ as a test function in  \eqref{eq:ausiliarioproblema} we have that
$\|w\|_{H^s_{\mathcal N, \vec 0}(\Omega)}\leq C \|h\|_{L^2(\Omega)}.$
Therefore taking a sequence $\{h_n\}_{n\in \mathbb N}$ bounded in $L^2(\Omega)$, we obtain that the sequence $w_n=\overset{\circ}{K}h_n$ is bounded in ${H^s_{\mathcal N, \vec 0}(\Omega)}$ as well, that is
\begin{equation}\label{eq:ajndkkhjkjdanomad}
\|w_n\|_{H^s_{\mathcal N, \vec 0}(\Omega)}\leq C,
\end{equation}
for some constant $C$ that does not depend on $n$. In particular from \eqref{eq:canzperunamic} it follows that
$$
\int_{\Omega}w_n^2\, dx+\int_{\Omega}\int_{\Omega}\frac{|\nabla w_n(x)-\nabla w_n(y)|^2}{|x-y|^{N+2\sigma}}\, dx dy\leq C.
$$
As we did in the proof of Lemma \ref{equivalencia} the previous inequality implies that $\|w_n\|^2_{W^{s,2}(\Omega)}<C$ ($s=1+\sigma$) so, since $W^{s,2}(\Omega)$ is compactly embedded in $L^{2}(\Omega)$, we deduce that, up to subsequences,  $\{w_n\}$ converges in $L^2(\Omega)$ as wanted.

 Moreover  the operator $\overset{\circ}{K}$ is self-adjoint. Indeed, taking {$h_1,h_2\in C_c^{\infty}(\Omega)$} and using the weak formulation
\eqref{eq:ausiliarioproblema}, for every $v\in {H^s_{\mathcal N, \vec 0}(\Omega)}$ we get that
\begin{equation}\label{eq:gen1}
\int_{\Omega}v\,Kh_1 dx+\frac{c_{N,\sigma}}{2} \int_{Q(\Omega)}\frac{(\nabla Kh_1(x)-\nabla Kh_1(y))\cdot(\nabla v(x)-\nabla v(y))}{|x-y|^{N+2\sigma}}\, dx dy=\int_{\Omega}h_1v dx
\end{equation}
and
\begin{equation}\label{eq:gen2}
\int_{\Omega}v\,Kh_2 dx+\frac{c_{N,\sigma}}{2} \int_{Q(\Omega)}\frac{(\nabla Kh_2(x)-\nabla Kh_2(y))\cdot(\nabla v(x)-\nabla v(y))}{|x-y|^{N+2\sigma}}\, dx dy=\int_{\Omega}h_2v.
\end{equation}
Using $v=Kh_2$ as test function in \eqref{eq:gen1} and $v=Kh_1$ as test function in \eqref{eq:gen2}, by \eqref{eq:comaaaavasc}, we deduce
\begin{equation}\label{eq:avbbbbbb}
\int_{\Omega}h_1\overset{\circ}{K}h_2\,dx=\int_{\Omega}h_2\overset{\circ}{K}h_1 \, dx.\end{equation}
Then  by a {density argument},  \eqref{eq:avbbbbbb} holds for $h_1,h_2\in L^2(\Omega)$ so this implies that $\overset{\circ}{K}$ is self-adjoint.
To conclude the proof in the homogeneous case we will show that
\begin{equation}\label{palma}
Ker\,(Id- \overset{\circ}{K})=\mathcal P_1(\mathbb R^N),
\end{equation}
that is, the Kernel of the operator $Id- \overset{\circ}{K}$ is the space of affine functions given in Definition \ref{def:affinfunctions}.
Let $p\in \mathcal P_1(\mathbb R^N)$, since $\nabla p$ is constant, firstly is clear that and observe that
$$
\int_{\Omega}pv\, dx+\frac{c_{N,\sigma}}{2} \int_{Q(\Omega)}\frac{(\nabla p(x)-\nabla p(y))\cdot(\nabla v(x)-\nabla v(y))}{|x-y|^{N+2\sigma}}\, dx dy=\int_{\Omega}p v\,dx.
$$
Moreover, using the definitions  \eqref{eq:neumann1} and \eqref{eq:neumann2}, it is also true that $\mathcal N^1_\sigma p (x)=0$ and $\mathcal N^2_\sigma p (x)=0$. Therefore $Kp (x)=p (x)$ in $\mathbb R^N$ and hence $\overset{\circ}{K}p(x)=p(x)$ in $\Omega$.
This shows that
$$\mathcal P_1(\mathbb R^N)\subset Ker\,(Id- \overset{\circ}{K}).$$
 The reverse inclusion is also true. In fact, taking now $w\in Ker\,(Id- \overset{\circ}{K})\subseteq L^2(\Omega)$, that is, $w=\overset{\circ}{K}w=Kw$ in $\Omega$, by the definition of $K$ we have that \begin{eqnarray}\label{eq:cuatrotrece}
&&\int_{\Omega}(Kw)v\, dx+\frac{c_{N,\sigma}}{2} \int_{Q(\Omega)}\frac{(\nabla Kw(x)(x)-\nabla Kw(x)(y))\cdot(\nabla v(x)-\nabla v(y))}{|x-y|^{N+2\sigma}}\, dx dy\\\nonumber
&&=\int_{\Omega}wv\, dx,\quad \forall v\in{H^s_{\mathcal N, \vec 0}(\Omega)}.
\end{eqnarray}
Then taking $v=w$ as a test function in \eqref{eq:cuatrotrece} we get
\[{\int_{Q(\Omega)}}\frac{|\nabla w(x)-\nabla w(y)|^2}{|x-y|^{N+2\sigma}}\, dx dy=0,
\]
which in particular implies that $w$ is a affine function, that is, $w\in \mathcal P_1(\mathbb R^N)$ as wanted.

Once we have proved \eqref{palma} applying the Fredholm alternative we obtain
\begin{center}
$Im(Id- \overset{\circ}{K})=Ker\,(Id- \overset{\circ}{K})^{\perp}=\mathcal P_1(\mathbb R^N)^{\perp},$
\end{center}
that is
$$Im(Id- \overset{\circ}{K})=\Big\{f\in L^2(\Omega)\,:\,(f,p)_{L^2(\Omega)}=0,\, p\in \mathcal P_1(\mathbb R^N)^{\perp}\Big\},$$
where by $(\cdot,\cdot)_{L^2(\Omega)}$ we denote the classical inner product in $L^2(\Omega)$.  By Theorem \ref{key_th} we have  that
\begin{equation}\label{fredh}
\mbox{the homogeneous problem \eqref{eq:p} has a solution if and only if $f\in \mathcal{P}_1(\mathbb R^N)^{\perp}$.}
\end{equation}
We can obtain again the same result by using the previous arguments:

Consider $f\in P_1(\mathbb R^N)^{\perp}=Im(Id- \overset{\circ}{K})$.  Then there exists $h\in L^2(\Omega)$ such that
\begin{equation}\label{eq:ffnumero}f=h-\overset{\circ}{K}h.\end{equation}
If we set $u={K}h$, then by construction, for every $v\in H^s_{\mathcal N, \vec 0}(\Omega)$, we get
\begin{equation}\label{eq:coffeinfaccia}
\int_{\Omega}uv\, dx+\frac{c_{N,\sigma}}{2} \int_{Q(\Omega)}\frac{(\nabla u(x)-\nabla u(y))\cdot(\nabla v(x)-\nabla v(y))}{|x-y|^{N+2\sigma}}\, dx dy=\int_{\Omega}hv\, dx,
\end{equation}
with
$$\mathcal N^1_\sigma u(x)=0,\qquad x\in \mathbb{R}^N\setminus\overline \Omega,\quad\text{and}\quad\mathcal N^2_\sigma u(x)=0,  \qquad x\in \partial \Omega.$$
Since $u=Kh=\overset{\circ}Kh$ in $\Omega$, from \eqref{eq:ffnumero} and \eqref{eq:coffeinfaccia}  it follows that
\begin{equation}\label{eq:solproblomogneo}
\begin{cases}
(-\Delta)^s u = f(x) &\text{in}\,\,\Omega\\
\mathcal N_\sigma^1 u  =0&\text{in}\,\,\mathbb{R}^N\setminus\overline\Omega\\
\mathcal N_\sigma^2 u=0 &\text{on}\,\,\partial\Omega,
\end{cases}\end{equation}
in the weak sense. Thus, $u$ is the desired solution. On the other hand if $u\in {H^s_{\mathcal N, \vec 0}(\Omega)}$ is a weak solution of \eqref{eq:solproblomogneo}, then we have
\[(-\Delta)^su+u=f+u,\quad \text{in}\,\, \Omega,\]
that is, by the definition of $K$, one has $u=K\big(u+f\big)$ in $\mathbb R^N$ and then  $u=\overset{\circ}K\big(u+f\big )$ in $\Omega$.
We deduce that
\[(Id-\overset{\circ}K)\big (u+f\big )=f, \quad \text{in}\,\, \Omega.\]
Then $f$ belongs to $Im(Id- \overset{\circ}{K})$ and, therefore,   it is such that $(f,p)_{L^2(\Omega)}=0$ for all functions $p\in \mathcal P_1(\mathbb R^N)$ as wanted.

This says that the nonhomogeneous case of problem  \eqref{eq:p} can be solved if we have an additional condition of the data, that is,
 if there exists $\psi$ sufficiently smooth such that
\[\mathcal{N}^1_\sigma(\psi)=g_1\,\, \text{in}\,\,\mathbb{R}^N\setminus\overline\Omega\qquad \text{and}\qquad \mathcal{N}^2_\sigma(\psi)=g_2 \,\,
\text{on}\,\,\partial \Omega.\]
If this is the case, then for   $f$,   $g_1$  and $g_2$  admissible data, we have that
\begin{equation}\nonumber
\int_{\Omega} fp \, dx +\int_{\mathbb R^N\setminus \Omega}\mathcal{N}^1_\sigma(\psi)p\, dx+\int_{\partial \Omega}\mathcal{N}^2_\sigma(\psi)p\, dS=0, \qquad \text{for all }\,\,  p\in \mathcal P_1(\mathbb R^N),
\end{equation}
By  Proposition \ref{pro:cuatrodiec} we obtain
\begin{equation}\label{eq:eccomiiii}
\int_{\Omega} \big(f -(-\Delta)^s\psi\big)p\,dx=0, \qquad \text{for all }\,\,  p\in \mathcal P_1(\mathbb R^N)
\end{equation}
Thus, by \eqref{fredh} and \eqref{eq:eccomiiii}, there exists a weak solution $\hat u$  to
\begin{equation}\nonumber
\begin{cases}
(-\Delta)^s \hat u = \hat f(x) &\text{in}\,\,\Omega\\
\mathcal N_\sigma^1 \hat u  =0&\text{in}\,\,\mathbb{R}^N\setminus\overline\Omega\\
\mathcal N_\sigma^2 \hat u=0 &\text{on}\,\,\partial\Omega,
\end{cases}\end{equation}
where
$$\hat f= f -(-\Delta)^s\psi\in Im(Id- \overset{\circ}{K}).$$
Therefore, defining $u:=\hat u +\psi$ we get that $u\in H^s_{\mathcal N, g_1, g_2}(\Omega)$ is a weak solution to
\begin{equation}\nonumber
\begin{cases}
(-\Delta)^s  u =  f(x) &\text{in}\,\,\Omega\\
\mathcal N_\sigma^1  u  =g_1&\text{in}\,\,\mathbb{R}^N\setminus\overline\Omega\\
\mathcal N_\sigma^2 u=g_2&\text{on}\,\,\partial\Omega.
\end{cases}\end{equation}
In both cases, homogeneous and non-homogeneous, the uniqueness up to a function $p\in \mathcal P_1(\mathbb R^N)$, follows easily by contradiction   using Lemma~\ref{eq:uniquuuueneessssss}.
\section{Spectral theory}
We will develop now the spectral theory associated to problem $\eqref{eq:p}$ using some general results established for compact operators. More precisely the complete description of the {structure} of the eigenvalues and eigenfunctions are given in the following
\begin{theorem}\label{spectral}
Let $\Omega\subset\mathbb{R}^{N}$ be a regular bounded domain. Then there exist a nondecreasing sequence $\{\lambda_i\}\geq 0$ and a sequence of functions $u_i:\mathbb{R}^{N}\to\mathbb{R}$ such that
$$
(\mathcal{P}_i)=\begin{cases}\label{ii}
(-\Delta)^s u_i = \lambda_i\, u_i &\text{in}\,\,\Omega\\
\mathcal N_\sigma^1 u_i  =0&\text{in}\,\,\mathbb{R}^N\setminus\overline\Omega\\
\mathcal N_\sigma^2 u_i=0 &\text{on}\,\,\partial\Omega,
\end{cases}
$$
Moreover the functions ${u_i}\big|_{\Omega}$ form a complete ortogonal system in the space $L^{2}(\Omega)$.
\end{theorem}
\begin{proof}
First of all we define de set
\begin{equation}\label{eq:trump}
L^{2,0}(\Omega):=\{u\in L^2(\Omega):\, \int_{\Omega}{u\, p\, dx}=0,\, {\forall}p\in\mathcal{P}_{1}(\mathbb{R}^{N})\},\end{equation}
{that contains the set $H^{s,0}_{\mathcal{N},\vec 0}(\Omega)$ defined in \eqref{noviembre_set} for $g_1=g_2=0.$}
Let us now consider the linear operator $T: L^{2,0}(\Omega)\to H^{s,0}_{\mathcal{N},\vec 0}(\Omega)$, such that $T(f):=u$ where $u$ is the (unique) solution of the problem
$$\begin{cases}
(-\Delta)^s u= f&\text{in}\,\,\Omega\\
\mathcal N_\sigma^1 u =0&\text{in}\,\,\mathbb{R}^N\setminus\overline\Omega\\
\mathcal N_\sigma^2 u=0 &\text{on}\,\,\partial\Omega,
\end{cases}$$
given in Theorem \ref{main}, recalling that $f$ satisfies \eqref{eq:unodiez}. We observe that the uniqueness come from the fact that $L^{2,0}(\Omega)$ is a closed subspace of $L^{2}(\Omega)$ and $L^{2,0}(\Omega)=\mathcal{P}_{1}(\mathbb{R}^{N})^{\perp}$.
As in the proof of Theorem \ref{main} we define the restriction  operator $\overset{\circ}{T}$ as
\begin{equation}\label{eq:colombia}
\overset{\circ}{T}f=Tf\big|_{\Omega}
\end{equation}
and therefore
$$\overset{\circ}T:L^{2,0}(\Omega) \to L^{2,0}(\Omega).$$
 With this notation {is clear that a function $u_i$ is a solution of problem $(\mathcal{P}_i)$ if and only if
\begin{equation}\label{autovaloresT}
\mbox{$u_i=T(\lambda_i u_i)=\lambda_i T(u_i)$,}
\end{equation}
therefore} it is possible to transform the question of the solvability of $(\mathcal{P}_i)$,  in the investigation of the eigenvalues and eigenfunctions of the operator $\overset{\circ}T$.
In order to use the well-know theory that establish the spectral properties of the operator, we will prove that $\overset{\circ} T$ is compact, self-adjoint and  positive in the Hilbert space $L^{2,0}(\Omega)$. Indeed using the weak formulation \eqref{eq:dbahgduvasc}, for every $f_1,f_2\in L^{2,0}(\Omega)$ and $v, \varphi,\in H^{s}_{\mathcal{N},\vec 0}(\Omega)$  it follows
\begin{eqnarray}\label{eq:dbahgduvasco}
&&\frac{c_{N,\sigma}}{2} \int_{Q(\Omega)}\frac{({\nabla Tf_1}(x)-\nabla Tf_1(y))\cdot(\nabla v(x)-\nabla v(y))}{|x-y|^{N+2\sigma}}\, dx dy=\int_{\Omega}f_1v\,dx,\\\nonumber
&&\frac{c_{N,\sigma}}{2} \int_{Q(\Omega)}\frac{(\nabla Tf_2(x)-\nabla Tf_2(y))\cdot(\nabla  \varphi(x)-\nabla  \varphi(y))}{|x-y|^{N+2\sigma}}\, dx dy=\int_{\Omega}f_2 \varphi\,dx.
\end{eqnarray}
Thus,  since $\overset{\circ}{T}f=Tf$ in $\Omega$, arguing as in equations \eqref{eq:gen1}-\eqref{eq:avbbbbbb}, we conclude that
$T$ is self-adjoint in $L^{2,0}(\Omega)$.
Further, using again \eqref{eq:dbahgduvasco}, it follows
$$(\overset{\circ}{T}f, f)_{L^2(\Omega)}=({T}f, f)_{L^2(\Omega)}=\frac{c_{N,\sigma}}{2} \int_{Q(\Omega)}\frac{|(\nabla Tf(x)-\nabla Tf(y))|^2}{|x-y|^{N+2\sigma}}\, dx dy\geq 0,$$
for every $f\in L^{2,0}(\Omega)$. {Moreover if $(\overset{\circ}{T}f, f)_{L^2(\Omega)}=0$ then $f\equiv 0$. Indeed if
$$(\overset{\circ}{T}f, f)_{L^2(\Omega)}=\int_{Q(\Omega)}\frac{|(\nabla Tf(x)-\nabla Tf(y))|^2}{|x-y|^{N+2\sigma}}\, dx dy=0,$$
then ${T}f\in\mathcal{P}_{1}(\mathbb{R}^{n})$ so that by \eqref{eq:dbahgduvasco} we deduce that $f\equiv 0$ as wanted.}
That is, the operator $\overset{\circ}{T}f$ is positive in $L^{2,0}(\Omega)$. Finally we will show that $T$ is compact in $L^{2,0}(\Omega)$. In fact, from \eqref{eq:dbahgduvasco}, with $Tf=u$ and $v=u$,   we get that
\begin{equation}\label{cero}
{\left(\int_{\Omega}\int_{\Omega}\frac{|\nabla u(x)-\nabla u(y)|^{2}}{|x-y|^{N+2\sigma}}\, dx dy\, \right)}\leq \int_{Q(\Omega)}\frac{|\nabla u(x)-\nabla u(y)|^{2}}{|x-y|^{N+2\sigma}}\, dx dy\leq \|f\|_{L^{2}(\Omega)}\|u\|_{L^{2}(\Omega)}.
\end{equation}
{Since the Poicar\'e inequality given in \eqref{poincare} is clearly satisfied by every $u\in H^{s,0}_{\mathcal{N},\vec 0}(\Omega)$}, from \eqref{cero} it follows that
\begin{equation}\label{laotra}
\left(\int_{\Omega}\int_\Omega\frac{|\nabla u(x)-\nabla u(y)|^{2}}{|x-y|^{N+2\sigma}}\, dx dy\right)^{\frac 12} \leq C\|f\|_{L^{2}(\Omega)}.
\end{equation}
Let us now consider $\{f_n\}$ a bounded sequence in $\in L^{2,0}(\Omega)$. By \eqref{poincare}  and \eqref{laotra},  repeating the arguments done to prove {\eqref{orden_noviembre} in Lemma \ref{equivalencia}, we infer that} $\{u_n=\overset{\circ} Tf_n\}$ is also bounded in the space $W^{s,2}(\Omega)$ ($s=1+\sigma$). Therefore since,  in particular, the inclusion $W^{s,2}(\Omega)\to L^{2}(\Omega)$ is compact, by subsequence, $u_n\to u$ in $L^{2}(\Omega)$.

Once we have proved that $\overset{\circ}{T}$ is compact, self-adjoint and a positive operator {in the separable space $L^{2,0}(\Omega)$} then (see for instance \cite[Theorem 3.8]{re}) the operator $\overset{\circ}{T}$ has a countable set of eigenvalues $\{\mu_{i}\}_{i\geq 2}$, all of them being positive. In particular  $$\mu_2\geq \mu_{3}\geq \ldots >0,\, \mbox{satisfying $\lim_{i\to\infty}{\mu_{i}}=0$.}$$
To the sequence $\{\mu_{i}\}_{i\geq 2}$ there corresponds a finite number of linearly independent eigenfuntions $\{{u}_{i}\}_{i\geq 2}$ that form a complete orthonormal system in $L^{2,0}(\Omega)$. 
Moreover, {as we noticed in \eqref{autovaloresT}}
\begin{eqnarray}\label{eq:lavorohopromesso}
\frac{c_{N,\sigma}}{2} \int_{Q(\Omega)}\frac{(\nabla u_i(x)-\nabla u_i(y))\cdot(\nabla \varphi (x)-\nabla  \varphi(y))}{|x-y|^{N+2\sigma}}\, dx dy&=&\frac{1}{\mu_i}\int_{\Omega}u_i \varphi\,dx\\\nonumber
&=:&\lambda_i\int_{\Omega}u_i \varphi\,dx.
\end{eqnarray}
Thus, by \eqref{eq:lavorohopromesso}  we finally infer  that
$$\{\lambda_i:=1/\mu_i,\, u_{i}\}_{i\geq 2},$$
form part of the suitable family of eigenfunctions and eigenvalues of $(\mathcal{P}_i)$ that we are looking for. To complete this family we observe that, by Lemma \ref{eq:uniquuuueneessssss}, it follows that   $\lambda_{1}=0$ is an eigenvalue with eigenfunctions
\begin{equation}\label{eq:flightttt}
\{u_{1,0}(x)=1,\, u_{1,1}(x)=x_1,\, \ldots u_{1,N}(x)=x_N\}.\end{equation}
Therefore, up to a reordering, we have obtained the sequence of eigenvalues
$$0=\lambda_1<\lambda_2\leq\ldots,\quad \lim_{i\to\infty}{\lambda_i}=\infty,$$
and its corresponding eigenfunctions $\{\{u_{1,j}\}^{N}_{j=0},\, u_{i}\}_{i\geq 2}$ that are a complete orthogonal system in $L^{2}(\Omega)$. Indeed,  as we have seen above,  the eigenfunctions $\{{u}_{i}\}_{i\geq 2}$ are  orthonormal w.r.t the $L^2$-scalar product and moreover, each  ${u}_i$  is orthogonal to  the subspace generated by the eigenfunctions~\eqref{eq:flightttt}, since the system  $\{{u}_i\}_{i\geq 2}$  belongs to  $L^{2,0}(\Omega)$.
Finally, to show that the orthogonal system is  maximal in $L^2(\Omega)$, let us consider $h\in L^{2}(\Omega)$ and  we define
$$\widetilde{h}=h-h_{1},$$
where $h_1$ is the orthogonal  projection (w.r.t. the $L^2$-scalar product) of $h$ in the subspace $\mathcal{P}_{1}(\Omega)$. Then $\widetilde{h}\in L^{2,0}(\Omega)$ and  $h_1=b_0+(b_j,x)\in\mathcal{P}_{1}(\Omega)$, for some  $b_j\in\mathbb{R},\, j=0,\,\ldots N$.
Since $\widetilde{h}\in L^{2,0}(\Omega)$ and $\{u_i\}_{i\geq 2}$ forms a complete system in $L^{2,0}(\Omega)$, then we obtain that
\begin{equation}\label{jestrella}
\lim_{k\to \infty}\left\|\widetilde{h}-\sum_{i=2}^{k} a_iu_i\right\|_{L^2(\Omega)}=0,
\end{equation}
for some real numbers $\{a_i\}$.  Moreover,
$$\widetilde{h}= h-\sum_{j=0}^{N}a_{1,j}\, u_{1,j}.$$ Thus, by \eqref{jestrella}, it follows that
$$\lim_{k\to \infty}\left\|{h}-\sum_{j=0}^{N}a_{1,j}u_{1,j}-\sum_{i=2}^{k} a_iu_i\right\|_{L^2(\Omega)}=0,$$
as wanted.
\end{proof}
\section{Further results and problems}
In this final section we describe in an informal way some further results and interesting open problems related with what we have seen in the previous sections.
\subsection{The  Neumann problem for $(-\Delta)^su $ in the case $s>2$}
In this subsection, using several integrations by parts (i.b.p., in short), we highlight the generalization in the higher-order case $s>2$ of  Proposition~\ref{pro:intbyparts2}, which was basic to define the Neumann problem. We write $s=m+\sigma$ and consider the case $m\geq 2$, even. {The case $m\geq3$ and odd can be obtained in the same way as in the proof of Proposition
\ref{pro:intbyparts2}. Therefore we skip this case.}

\

{\bf Case:  $m\geq 2$, even.} Let us define the natural Neumann conditions that come from the following non local  higher-order integration by parts formula. For  suitable $v\in \mathcal{S}(\mathbb{R}^{N})$, we define
\begin{equation}\label{eq:allanimaa}
\mathcal N^{1,\,{(i-1)}}_\sigma v(x){:=\Delta^{i-1}\big((-\Delta)^{\sigma}_{\mathbb{R}^N\setminus \Omega}(\Delta^{\frac m2} v)}\big)(x),  \qquad x\in \partial \Omega \quad \text{and}\quad i=1,2,\ldots, m/2
\end{equation}
and
\begin{equation}\label{eq:allanimaa1}
\mathcal N^{2,\,{(i-1)}}_\sigma v(x){:=\frac{\partial}{\partial\nu}\Delta^{i-1}\big((-\Delta)^{\sigma}_{\mathbb{R}^N\setminus \Omega}(\Delta^{\frac m2} v)}\big)(x),  \qquad x\in \partial \Omega \quad \text{and}\quad i=1,2,\ldots, m/2,
\end{equation}
with $\nu$ denoting the unit outer normal to $\partial \Omega$. {With these definitions, {\em for suitable} $u,v\in \mathcal{S}(\mathbb{R}^{N})$ (in particular with $u$ satisfying   similar hypotheses  to \eqref{condition3}) it  can be shown   the following}
\begin{eqnarray}\nonumber
&&\frac {c_{N,\sigma}}{2}\int_{Q(\Omega)}\frac{(\Delta^{\frac m2} u(x)-\Delta^{\frac m2} u(y))(\Delta^{\frac m2} v(x)-\Delta^{\frac m2} v(y))}{|x-y|^{N+2\sigma}}\,dx\,dy\\\nonumber
&&=\int_{\Omega}v\,(-\Delta)^su\, dx +\int_{\mathbb{R}^N\setminus \Omega} \Delta^{\frac m2}\mathcal N_\sigma(\Delta^{\frac m2}u(x))\, v(x)\, dx\\\nonumber
&&+\sum_{i=1}^{m/2}\int_{\partial \Omega}\frac{\partial}{\partial\nu}\left(\Delta^{\frac{m-2i}{2}}v(x)\right)\, \mathcal N^{1,\,{(i-1)}}_\sigma u(x)\, dS-\sum_{i=1}^{m/2}\int_{\partial \Omega}\mathcal N^{2,\,{(i-1)}}_\sigma u(x)\,\Delta^{\frac{m-2i}{2}}v(x)\, dS.
\end{eqnarray}
{In fact, roughly speaking,}  denoting by $\nu$ the unit outer normal field to the boundary  $\partial \Omega$ and integrating twice by parts we obtain
\begin{eqnarray}\label{eq:cicoooo}
&&\frac {c_{N,\sigma}}{2}\int_{Q(\Omega)}\frac{(\Delta^{\frac m2} u(x)-\Delta^{\frac m2} u(y))(\Delta^{\frac m2} v(x)-\Delta^{\frac m2} v(y))}{|x-y|^{N+2\sigma}}\,dx\,dy\\\nonumber
&&=c_{N,\sigma}\int_{\Omega}\Delta^{\frac m2} v(x)\int_{\mathbb R^N}\frac{(\Delta^{\frac m2} u(x)-\Delta^{\frac m2} u(y))}{|x-y|^{N+2\sigma}} \, dy\, dx
\\\nonumber &&+ c_{N,\sigma}\int_{\mathcal C \Omega}\Delta^{\frac m2} v(x)\int_{\Omega}\frac{(\Delta^{\frac m2} u(x)-\Delta^{\frac m2} u(y))}{|x-y|^{N+2\sigma}}\, dy\, dx\\\nonumber
&&\overset{\text{1th i.b.p.}}{=} -\int_{\Omega}\nabla (\Delta^{\frac{m-2}{2}}v(x))\cdot \nabla \left (c_{N,\sigma}\int_{\mathbb R^N}\frac{(\Delta^{\frac m2} u(x)-\Delta^{\frac m2} u(y))}{|x-y|^{N+2\sigma}} \, dy\right)\, dx\\\nonumber
&&-\int_{\mathbb{R}^N\setminus \Omega}\nabla (\Delta^{\frac{m-2}{2}}v(x))\cdot \nabla \left (c_{N,\sigma}\int_{\Omega}\frac{(\Delta^{\frac m2} u(x)-\Delta^{\frac m2} u(y))}{|x-y|^{N+2\sigma}} \, dy\right)\, dx \\\nonumber
&&+\int_{\partial \Omega}\frac{\partial}{\partial\nu}\left(\Delta^{\frac{m-2}{2}}v(x)\right)\left (c_{N,\sigma}\int_{\mathbb{R}^N\setminus \Omega}\frac{(\Delta^{\frac m2} u(x)-\Delta^{\frac m2} u(y))}{|x-y|^{N+2\sigma}} \, dy\right)\, dS\\\nonumber
&&\overset{\text{2th i.b.p}}{=}\int_{\Omega} \Delta^{\frac{m-2}{2}}v(x)\,\Delta \left (c_{N,\sigma}\int_{\mathbb R^N}\frac{(\Delta^{\frac m2} u(x)-\Delta^{\frac m2} u(y))}{|x-y|^{N+2\sigma}} \, dy\right)\, dx
\\\nonumber
&&+\int_{\mathbb{R}^N\setminus \Omega}\Delta^{\frac{m-2}{2}}v(x)\,\Delta \left (c_{N,\sigma}\int_{\Omega}\frac{(\Delta^{\frac m2} u(x)-\Delta^{\frac m2} u(y))}{|x-y|^{N+2\sigma}} \, dy\right)\, dx
\\\nonumber
&&+\int_{\partial \Omega}\frac{\partial}{\partial\nu}\left(\Delta^{\frac{m-2}{2}}v(x)\right)\left (c_{N,\sigma}\int_{\mathbb{R}^N\setminus \Omega}\frac{(\Delta^{\frac m2} u(x)-\Delta^{\frac m2} u(y))}{|x-y|^{N+2\sigma}} \, dy\right)\, dS\\\nonumber
&&- \int_{\partial \Omega}\Delta^{\frac{m-2}{2}}v(x)\frac{\partial}{\partial \nu}\left (c_{N,\sigma}\int_{\mathbb{R}^N\setminus \Omega}\frac{(\Delta^{\frac m2} u(x)-\Delta^{\frac m2} u(y))}{|x-y|^{N+2\sigma}} \, dy\right)\, dS.
\end{eqnarray}
Then if we continue to integrate by parts, as we did in  \eqref{eq:cicoooo}, after $m/2$ steps we get
\begin{eqnarray}\nonumber
&&\frac {c_{N,\sigma}}{2}\int_{Q(\Omega)}\frac{(\Delta^{\frac m2} u(x)-\Delta^{\frac m2} u(y))(\Delta^{\frac m2} v(x)-\Delta^{\frac m2} v(y))}{|x-y|^{N+2\sigma}}\,dx\,dy\\\nonumber
&&\vdots\\\nonumber
&&\overset{\text{mth i.b.p}}{=}\int_{\Omega} v(x)\Delta^{\frac m2} \left (c_{N,\sigma}\int_{\mathbb R^N}\frac{(\Delta^{\frac m2} u(x)-\Delta^{\frac m2} u(y))}{|x-y|^{N+2\sigma}} \, dy\right)\, dx\\\nonumber
&&+\int_{\mathbb{R}^N\setminus \Omega} v(x)\Delta^{\frac m2} \left (c_{N,\sigma}\int_{\Omega}\frac{(\Delta^{\frac m2} u(x)-\Delta^{\frac m2} u(y))}{|x-y|^{N+2\sigma}} \, dy\right)\, dx\\\nonumber
&&+ \sum_{i=1}^{m/2}\int_{\partial \Omega}\frac{\partial}{\partial\nu}\left(\Delta^{\frac{m-2i}{2}}v(x)\right)\, \Delta^{i-1}\left (c_{N,\sigma}\int_{\mathbb{R}^N\setminus \Omega}\frac{(\Delta^{\frac m2} u(x)-\Delta^{\frac m2} u(y))}{|x-y|^{N+2\sigma}} \, dy\right)\, dS\\\nonumber
&&-\sum_{i=1}^{m/2}\int_{\partial \Omega}\Delta^{\frac{m-2i}{2}}v(x)\frac{\partial}{\partial \nu}\Delta^{i-1}\left (c_{N,\sigma}\int_{\mathbb{R}^N\setminus \Omega}\frac{(\Delta^{\frac m2} u(x)-\Delta^{\frac m2} u(y))}{|x-y|^{N+2\sigma}} \, dy\right)\, dS
\end{eqnarray}
and then we obtain the conclusion using Proposition \ref{pro:equivfraclapl} and equations  \eqref{eq:neumannenrico}, \eqref{eq:allanimaa} and \eqref{eq:allanimaa1}.
\subsection{A semilinear Neumann problem and some open questions}
Consider the problem
\begin{equation}\label{semilinearproblem}
\left\{
\begin{array}{rcll}
d(-\Delta)^s u+u&=&|u|^{p-1}u \quad & x\in \Omega\\
\mathcal{N}_s(u)&=&0\quad & x\in \mathbb{R}^{{N}}\setminus \Omega,
\end{array}
\right.
\end{equation}
where $0<s<1$, $\Omega\subset\mathbb{R}^{{N}}$ is a smooth bounded domain and the diffusion coefficient $d$ is positive.

For the classical Laplacian this problem was  deeply analyzed  by Lin, Ni and Takagi in their classical paper \cite{LNT} where the reader can also see  the motivations of this model in the local case.

First of all we notice that $v_0 =0$ and $v=\pm 1$ are the unique possible constant solutions of problem \eqref{semilinearproblem}. Therefore we will be interested in finding nontrivial solutions.
It is clear that if $1<p< 2^*_s$  this is equivalent to  look for non-constant critical points of the  energy functional,
\begin{equation}\label{energys=1}
J_d(u)=\frac {{d}}{2}\int\int_{Q(\Omega)}\dfrac{|u(x)-u(y)|^2}{|x-y|^{N+2s}}dxdy+\frac 12\int_\Omega u^2dx-\frac 1{p+1}\int_\Omega |u|^{p+1} dx,
\end{equation}
where
$$2^*_s=\left\{
\begin{array}{lll}
&\dfrac{N+2s}{N-2s},\quad &N>2s\\
&+\infty,\quad  &N\le 2s.
\end{array}
\right.
$$
is the critical fractional exponent. Notice that $J_d$ is well defined in
\begin{equation}\label{spacebase1}
{H}^s(\Omega)=\left\{\, u:\mathbb{R}^N\rightarrow\mathbb{R}\,|\, u \hbox{ measurable,  } \,\, {
\iint_{Q(\Omega)} \dfrac{|u(x)-u(y)|^2}{|x-y|^{N+2\sigma}}dxdy}<+\infty\right\},
\end{equation}

Therefore the main result is the following
\begin{theorem}\label{existence}
There exists a nontrivial nonconstant solution, $u_d$, to problem \eqref{semilinearproblem}, provided $d$ is sufficiently small.
\end{theorem}
\begin{proof}
Following closely the arguments done in \cite{LNT}, we can use the \textit{Mountain-Pass Lemma} by Ambrosetti-Rabinowitz, \cite{AR}, in order to find critical points of $J_d$. Indeed, it is easy to check that the geometry of the  \textit{Mountain-Pass Lemma}  holds for all $d>0$,  that is, there exists $\rho$  such that
$J_d(u)>0$ for all $0<\|u\|\le \rho$,  $J_d(u)>\beta>0$, $\|u\|=\rho$ and there exists $v$, with $\|v\|>\rho$ such that $J_d({v})<0$.
 Moreover any Palais-Smale sequence is bounded and by the Rellich Compactness Theorem admits a convergent subsequence. Let us assume now that
\begin{equation}\label{emma}
\mbox{$\exists$ $\phi_{d}\in H^s(\Omega)$, and $t_1,\widetilde{C}>0$ such that $J_d(t_1\phi_d)=0$ and $J_d(t\phi_d)<\widetilde{C}d^{\frac{N}{2s}}$, $0\leq t\leq t_1$.}
\end{equation}
Then if $\Gamma=\{\gamma\in \mathcal{C}\left([0,1],H^{s}(\Omega)\right)\,|\, \gamma(0)=0, \gamma(1)={t_1\phi_d}\}$, it follows that the minimax value
$$c_d=\inf\limits_{\gamma\in \Gamma}\max\limits_{[0,1]}J_d(\gamma(t)){\geq\beta>0=J_d(0)},$$
is a crititical value of $J_d$, {that is, there exists a solution $\widetilde{u}$ such that $J_{d}(\widetilde{u})=c_{d}$. Moreover since, by \eqref{emma}, taking $d$ small enough,
$$J_d(\widetilde{u})=c_d\leq max_{[0,t_1]} J_d(t\phi_d)\leq \widetilde{C}d^{N/2s}< \left(\frac 12-\frac{1}{p+1}\right)|\Omega|=J_d(1)=J_d(-1),$$
we conclude that  in the set $J_d^{-1}(c_d)$ there is some nonconstant critical point.}

To prove \eqref{emma} let us consider
$$\phi(x)=\begin{cases}
(1-|x|), \quad |x|<1\\
0, \quad |x|\ge 1,
\end{cases}
$$
and define $\phi_d (x)=d^{-\frac {N}{2s}}\phi\left(\frac{x}{d^{\frac {1}{2s}}}\right)$.
Assume without lost of generality that $0\in\Omega$ and take $d$ small enough in such a way that the ball of radius $d^{\frac {1}{2s}}$  is contained in $\Omega$.

We have by a direct calculation one can check that
$${\frac 12\int\int_{Q(B_{d^{\frac {1}{2s}}})}\dfrac{|\phi_d(x)-\phi_d(y)|^2}{|x-y|^{N+2s}}dxdy\le C(N) d^{-\left(1+\frac N{2s}\right)},} \quad
||\phi_d||_q^q= C_q(N)  d^{\frac N{2s}(1-q)}
$$
Therefore, for constants depending only on the dimension,
$$g(t):=J_d(t\phi_d)=c_1\dfrac 12 d^{-\frac N{2s}}t^2-c_2\dfrac 1{p+1}d^{-\frac N{2s}p}t^{p+1}.$$
It is easy to check that there exists $t_1>0$ such that $g(t_1)=0$. Moreover $g$
verifies that its maximum for $t>0$ is attained in  $t_0=(\frac{c_1}{c_2})^{\frac1{p-1}} d^{\frac N{2s}}<t_1$, {that is}
$$g(t)=J_d(t\phi_d)\le J_d(t_0\phi_d)\le C  d^{\frac N{2s}},$$
{and \eqref{emma} follows as wanted.}
\end{proof}
\begin{remark} The higher order Neumann semilinear problem can be studied  in a similar way. To be precise we consider the case $s=1+\sigma$, $0<\sigma<1$ and leave to the reader the details for  the interval $s>2$. Consider the problem
\begin{equation}\label{eq:p}\tag{$\mathcal P$}
\begin{cases}
d(-\Delta)^s u+u = |u|^{p-1}u &\text{in}\,\,\Omega, \quad 1<s<2\\
\mathcal N_\sigma^1 u  =0&\text{in}\,\,\mathbb{R}^N\setminus\overline\Omega=: \mathcal C\overline\Omega\\
\mathcal N_\sigma^2 u=0 &\text{on}\,\,\partial\Omega,
\end{cases}\end{equation}
for $d>0$, $s=1+\sigma$, $\sigma>0$ and $1< p< 2^*_s$.
The energy functional now is
\begin{equation}\label{energys>1}
J_d(u)=\frac d2\int\int_{Q(\Omega)}\dfrac{|\nabla u(x)-\nabla u(y)|^2}{|x-y|^{N+2\sigma}}dxdy+\frac 12\int_\Omega u^2dx-\frac 1{p+1}\int_\Omega |u|^{p+1} dx.
\end{equation}
Observe that,  as in the case $s<1$, $J_d$ verifies the geometrical and compactness hypotheses to apply the \textit{Mountain Pass Theorem } so taking, for instance,  $\phi(x)=(1-|x|^2)^2_+$ it follows that
$$J_d(\phi_d)\le C d^{\frac{N}{2s}}.$$
Therefore a similar argument as above shows that, for $d$ small enough, the mountain pass critical point is nonconstant.
\end{remark}
\subsubsection{Some open questions}
Among others, the following questions seem to be open and interesting to solve.
\begin{enumerate}
\item Asymptotic behavior of the nonconstant solutions when $d\to 0$, $0<s$.  The local case $s=1$ this problem was  studied for positive solutions
in the pioneering paper \cite{LNT}, where a concentration phenomenon appears in the point of maximum curvature of  $\partial \Omega$.
As far as we know, this result should be new in the local case $s=2$ or higher integer order.
 \item Study of the critical case $p=2^*$ and the behavior of the nonconstant solutions. The local case, $s=1$, was studied in \cite{APY}.
\end{enumerate}

\subsection{A Neumann condition for the  p-Laplacian  operator $(-\Delta)_p^s$ in the standard nonlocal case $\bf{0<s<1 }$}\label{sec:plaplaciano}
Using the variational approach  for the higher order operator developed in Section \ref{sec:variational}, we define a Neumann problem for the nonlinear p-Laplacian nonlocal operator that, to the best of our knowledge, it has not been studied up to now.  Throughout all this section, let us suppose $p\in (1,\infty)$,  $s\in (0,1)$ and let $\Omega\subset {\mathbb R}^N$ be a smooth bounded domain with $N>sp$. For smooth functions,  we define  the fractional $p$-Laplacian operator $(-\Delta)_p^s$ {(see for instance \cite{p1, p2, p3} and the references therein)}, as
\begin{equation}\label{eq:plaplace}
(- \Delta)^s_p\, u(x):= \lim_{\varepsilon \searrow 0} \int_{\mathbb{R}^N \setminus B_\varepsilon(x)}
\frac{|u(x) - u(y)|^{p-2}\, (u(x) - u(y))}{|x - y|^{N+s\,p}}\, dy, \qquad x \in \mathbb{R}^N.
\end{equation}
Moreover for $\mathcal{S}(\mathbb{R}^{N})$, define
\begin{equation}\label{eq:neumannenricop}
\mathcal N_{s,p} v(x):=\int_{\Omega}\frac{|v(x)-v(y)|^{p-2}(v(x)-v(y))}{|x-y|^{N+2s}}\,dy, \qquad x\in \mathbb{R}^N\setminus\overline \Omega,
\end{equation}
namely the non local Neumann condition in the case of the non local p-Laplace operator.  The equation \eqref{eq:neumannenricop} represents  the counterpart in the non local case,  of the the local Neumann condition $|\nabla u|^{p-2}\partial_\nu u$, i.e. the normal component of the flux across the boundary.
Following the proofs of the Proposition \ref{pro:intbyparts1} and of the Proposition \ref{pro:intbyparts2},  using  definitions  \eqref{eq:neumannenricop} and \eqref{eq:plaplace} it can be proved the following
\begin{theorem}\label{eq:nnmiservealtro}
Let $u,v\in \mathcal{S}(\mathbb{R}^{N})$. Then
\begin{equation}\nonumber
\int_{\Omega}(-\Delta)_p^su\,dx=-\int_{\mathbb R^N\setminus \Omega} \mathcal N_{s,p} u \,dx
\end{equation}
and
\begin{eqnarray}\nonumber
&&\frac{1}{2}\int_{Q(\Omega)}\frac{|u(x)-u(y)|^{p-2}(u(x)- u(y))( v(x)-v(y))}{|x-y|^{N+2s}}\, dx\,dy\\\nonumber
&&=\int_{\Omega}v\, (-\Delta)_p^su\,  dx +\int_{\mathbb{R}^N\setminus \Omega} v\, \mathcal {N}_{s,p} u \, dx.
\end{eqnarray}
\end{theorem}
Theorem \ref{eq:nnmiservealtro} suggests, as we did in Section \ref{se:weakvaria},  the idea of which should be the correct weak formulation of the nonlocal $p$-Neumann problem in the case $0<s<1$, i.e. it gives the good candidate for the weak form of the p-Laplace operator \eqref{eq:plaplace}.  Now we can give the following
\begin{definition} Let $g\in L^1(\mathbb R^N\setminus\Omega)$, we set
$$
W^{s,p}_{\mathcal N(g)}(\Omega)=\Big\{u:\mathbb R^N\rightarrow {\mathbb R}\,:\, \, u
{\mbox{ measurable }}\quad \text{and}\quad  \|u\|_{W^{s,p}_{\mathcal N(g)}{(\Omega)}}< +\infty\Big\},
$$
where
\begin{equation}\label{eq:pnorm}
\|u\|_{W^{s,p}_{\mathcal N(g)}(\Omega)}=\left (\int_\Omega |u|^p\,dx+ \iint_{Q(\Omega)}\frac{|u(x)-u(y)|^p}{|x-y|^{N+sp}}\, dx dy+ \int_{\mathcal C \Omega}|g||u|^p\, dx\right)^{\frac 1p}.
\end{equation}
\end{definition}
Following the ideas done in \cite[Proposition 3.1]{DRoV} we can be proved the next
\begin{proposition}\label{pro:reflex}
$W^{s,p}_{\mathcal N(g)}(\Omega)$ is a reflexive Banach space.
\end{proposition}
\begin{proof}
We sketch the proof. We can readily  check that \eqref{eq:pnorm} is a norm and, arguing  as in \cite[Proposition 3.1]{DRoV}, that  $W^{s,p}_{\mathcal N(g)}(\Omega)$ is a Banach space.  To prove that it is reflexive let us define the space $\mathcal A= L^p(Q(\Omega), dxdy)\times L^p(\Omega,dx) \times L^p(\mathbb{R}^N\setminus\Omega,|g|dx).$ By standard results, this product space is reflexive. Then the operator $T:W^{s,p}_{\mathcal N(g)}(\Omega) \rightarrow \mathcal A$ defined as
\[T u =\left[\frac{|u(x)-u(y)|}{|x-y|^{\frac Np+s}}\chi_{Q(\Omega)}(x,y),\,u \chi_{\Omega}, \, u \chi_{\mathbb R^N\setminus\Omega}\right],\]
where $\chi_\mathcal S(\cdot)$ denotes the characteristic function of a measurable set $\mathcal S$, is an isometry from $W^{s,p}_{\mathcal N(g)}(\Omega)$ into $\mathcal A$ (the space $\mathcal A$ is also equipped  with the norm in \eqref{eq:pnorm}).  Thus, since  $W^{s,p}_{\mathcal N(g)}(\Omega)$ is a Banach space, $T(W^{s,p}_{\mathcal N(g)}(\Omega))$ is a closed subspace of $\mathcal A$, so that, $T(W^{s,p}_{\mathcal N(g)}(\Omega))$  is reflexive and therefore, using the fact that $T$ is an isometry, $W^{s,p}_{\mathcal N(g)}(\Omega)$ as well.
\end{proof}
Thanks to the previous result, we can use the variational arguments developed in, for instance \cite{dibenedetto}, to get the existence and uniqueness result for the $(-\Delta)^s_{p}$ operator, $0<s<1$. That is, the following
\begin{theorem}\label{main2}
Let $\Omega\subset\mathbb R^N$ a bounded $C^1$ domain and let us suppose that $f\in L^{p'}(\Omega)$, $\frac{1}{p}+\frac{1}{p'}=1$, and {$g\in L^{\infty}_c(\mathbb{R}^N\setminus \Omega)$}.
Then, the problem
$$
\begin{cases}
(-\Delta)^s_p u = f(x) &\text{in}\,\,\Omega\\
\mathcal N_{s,p} u  =g&\text{in}\,\,\mathbb{R}^N\setminus\overline\Omega,
\end{cases}
$$
has a weak solution, that is,
$$\frac{1}{2}\int_{Q(\Omega)}\frac{|u(x)-u(y)|^{p-2}(u(x)- u(y))( v(x)-v(y))}{|x-y|^{N+2s}}\, dx\,dy=\int_{\Omega}f\, v dx +\int_{\mathbb{R}^N\setminus \Omega} g\, v \, dx,\, \forall v \in W^{s,p}_{\mathcal N(g)}(\Omega),$$
 if and only if the following compatibility condition holds
\begin{equation}\label{eq:unodiez_p}
\int_{\Omega} f \, dx +\int_{\mathbb R^N\setminus \Omega}g \, dx=0.
\end{equation}
Moreover, if \eqref{eq:unodiez_p} holds,  the solution is unique up to a constant $c\in \mathbb R^N.$
\end{theorem}

The proof of the previous result can be done using the same minimization techniques developed in the proof Theorem \ref{key_th} under hypothesis $(\mathcal A_{(f,g_1,g_2)})-Case \, A$ once we prove a Poincar\'e-type inequality in the Banach space
\begin{equation}\label{eq:defspazlasecondvoltnat}
\widetilde{W}^{s,p}_{\mathcal N(g)}(\Omega):=\{u\in W^{s,p}_{\mathcal N(g)}(\Omega):\,  \int_{\Omega} u\, dx=0\}.
\end{equation}
This inequality will allow us to affirm that the norm in $\widetilde{W}^{s,p}_{\mathcal N(g)}(\Omega)$ defined as
$$\|u\|^p_{\widetilde{W}^{s,p}_{\mathcal N(g)}(\Omega)}:=\int_{Q(\Omega)}\frac{|u(x)- u(y)|^p}{|x-y|^{N+sp}}\, dx dy,$$
is equivalent to the one in ${W}^{s,p}_{\mathcal N(g)}(\Omega)$ given in \eqref{eq:pnorm}.
\begin{lemma}\label{pro:poincplapl}
For every
$v\in \widetilde{W}^{s,p}_{\mathcal N(g)}(\Omega)$ if $g\in L^{\infty}_c(\mathbb{R}^N\setminus \Omega)$ then the following Poincar\'e-type inequality holds
\begin{equation}\nonumber
\int_{\Omega}|v|^p\, dx+\int_{\mathbb{R}^N\setminus \Omega}|g||v|^p\, dx\leq C(N,s,\Omega)\iint_{Q(\Omega)}\frac{|v(x)- v(y)|^p}{|x-y|^{N+sp}}\, dx dy.
\end{equation}
\end{lemma}
\begin{proof}
Following the proof of \eqref{poincare} let us suppose, by contradiction, that there exists, up to a renormalization, a sequence $\{v_k\}\subset \widetilde{W}^{s,p}_{\mathcal N(g)}(\Omega)$ such that
\begin{equation}\label{pplol}
\int_{\Omega}|v_k|^p\, dx+\int_{\mathbb{R}^N\setminus \Omega}|g||v_k|^p\, dx=1\,\, \mbox{and}\,\, \iint_{Q(\Omega)}\frac{|v_k(x)-v_k(y)|^{p}}{|x-y|^{N+sp}}\, dx dy<\frac{1}{k}.
\end{equation}
Using now that the embedding of $W^{s,p}(\Omega)$ is compact (see \cite{DnPV}), it follows that, up to subsequence, there exists  $v\in L^p(\Omega)$ such that
\begin{equation}\label{porq}
v_{k}\rightarrow v,\qquad\text{in}\,\, L^p(\Omega).
\end{equation}
Moreover if we take a ball $B$ centered at the origin with $\Omega \subset B$  we get by elementary inequalities that
\begin{equation}\label{eq:cauchyyyyyyseqvk}
\fint_B |v_k|^p\,dx\leq C(s,p,B, N, \Omega)\fint_{B}\fint_{\Omega}\frac{|v_k(x)-v_k(y)|^{p}}{|x-y|^{N+sp}}\, dx dy+\fint_\Omega |v_k|^p\,dx
\end{equation}
By \eqref{pplol}- \eqref{eq:cauchyyyyyyseqvk} we deduce that for all $\varepsilon>0$, there exists $\overline k$ such that  for all $m,k>\overline k$
\[\int_{B}|v_k-v_m|^p\, dx< \varepsilon,\]
namely in particular $v_k$ is a Cauchy sequence in $L^p(B)$ and therefore, up to a subsequence, $v_k$ converges to some $v$ in $L^p(B)$ and a.e. in B. Passing to the limit in \eqref{pplol}, by  the lower semicontinuity of the norm w.r.t. the weak convergence, on one hand we get that $v$ must be a constant in $\mathbb{R}^N$ and on the other hand that
\[\int_{\Omega}|v|^p\, dx+\int_{\mathbb{R}^N\setminus \Omega}|g||v|^p\, dx=1,\]
that is a contradiction with \eqref{eq:defspazlasecondvoltnat}.
\end{proof}
Now we can give the
\begin{proof}[Proof of Theorem \ref{main2}]
For every $u\in {W}^{s,p}_{\mathcal N(g)}(\Omega)$ we define the nonlinear functional
$$J_p(u):= \frac1p \int_{Q(\Omega)}\frac{|\nabla u(x)-\nabla u(y)|^p}{|x-y|^{N+sp}}\, dx dy-\int_{\Omega}f{u}\,dx-\int_{\mathbb R^N\setminus \Omega}g{u}\, dx.$$
We note that, by the compatibility condition \eqref{eq:unodiez_p}, it follows that  $J(u)=J(u-\overline u)$, where
$$\overline u =\frac{1}{|\Omega|}\int_{\Omega}u\, dx.$$
Therefore $J_p$ can be defined in the space $\widetilde{W}^{s,p}_{\mathcal N(g)}(\Omega)$. Following the proof of Theorem \ref{key_th} and the Poincar\'e inequality given in Lemma \ref{pro:poincplapl} the result follows.
\end{proof}

\subsubsection{Some open questions}{Among many other possible choices, we think that
it would  be very interesting to find a natural Neumann condition for the operator
\begin{equation}\label{eq:plaplace1+}
(- \Delta)^s_p\, u(x):= -\operatorname{div} \left (\lim_{\varepsilon \searrow 0} \int_{\mathbb{R}^N \setminus B_\varepsilon(x)}
\frac{|\nabla u(x) - \nabla u(y)|^{p-2}\, (\nabla u(x) -\nabla u(y))}{|x - y|^{N+\sigma\,p}}\, dy, \qquad x \in \mathbb{R}^N\right),
\end{equation}
where $s=1+\sigma$.
}


\begin{thebibliography}{10}

\bibitem{AJS}N. Abatangelo, S. Jarohs,  A. Salda\~na,
\newblock On the maximum principle for higher-order fractional
Laplacians. {\em Preprint},  arXiv:1607.00929.

\bibitem{APY} Adimurthi, F. Pacella, S. L. Yadava,
\newblock Interaction between the geometry of the boundary and positive solutions of a semilinear Neumann problem with critical nonlinearity. \emph{J. Functional Anal.} 113 (1993), no. 2 318-350.

\bibitem{AR}  A. Ambrosetti, P. H. Rabinowitz,
\newblock Dual variational methods in critical point theory and applications.\emph{J. Functional Anal.} {14} (1973), 349-381.



\bibitem{LA} J. M. Arrieta, P. D. Lamberti,
\newblock Higher order elliptic operators on variable domains.
Stability results and boundary oscillations for intermediate
problems. {\em J. Differential Equations}, 263 (2017), no.7, 4222-4266.


\bibitem{BCGJ} G. Barles, E. Chasseigne, C. Georgelin, E. Jakobsen,
\newblock On Neumann type problems for nonlocal equations
in a half space. {\em Trans. Amer. Math. Soc.} 366 (2014), no. 9, 4873-4917.

\bibitem{BGJ} G. Barles, C. Georgelin, E. Jakobsen,
\newblock On Neumann and oblique derivatives boundary conditions for
nonlocal elliptic equations. {\em J. Differential Equations}, 256 (2014), 1368-1394.


\bibitem{BDGQ} B. Barrios, L. Del Pezzo, J. Garcia-Melian, A. Quaas,
\newblock Monotonicity of solutions for some nonlocal elliptic problems in half-spaces. {\em Calc. Var. Partial Differential Equations}, 56 (2017), 56-39.

\bibitem{BDGQ2} B. Barrios,L. Del Pezzo, J. Garcia-Melian, A. Quaas,
\newblock Symmetry results in the half-space for a semi-linear fractional Laplace equation through a one-dimensional
analysis.  {\em Ann. Mat. Pura Appl. (4)}, to appear.

\bibitem{BMS} B.~Barrios, L.~ Montoro, B.~Sciunzi, \newblock
On the moving plane method for nonlocal problems in bounded domains. {\em J. Anal. Math.}, to appear.

\bibitem{p1} B. Barrios,  I. Peral, S. Vita, \newblock Some remarks about the summability of nonlocal nonlinear problems. {\em Adv. Nonlinear Anal.}  4  (2015),  no. 2, 91-107.


\bibitem{BBC} K. Bogdan, K. Burdzy, Z. Q. Chen,
\newblock Censored stable processes. {\em Probab. Theory Relat. Fields} 127 (2003), 89-152.


\bibitem{BL}V. Burenkov, P. D.  Lamberti,
\newblock Spectral stability of higher order uniformly elliptic operators. {\em Sobolev spaces in mathematics. II},
Int. Math. Ser. (N. Y.), 9 (2009), 69-102.



\bibitem{C1} L.M. Chasman,
\newblock An isoperimetric inequality for fundamental tones of free plates. {\em Comm. Math. Phys.}, 303 (2011), no. 2, 421-449.


\bibitem{CK} Z. Q. Chen, P. Kim,
\newblock Green function estimate for censored stable processes. {\em Probab. Theory Relat.
Fields} 124 (2002), 595-610.


\bibitem{CERW} C. Cortazar, M. Elgueta, J. Rossi, N. Wolanski, \newblock Boundary fluxes for nonlocal diffusion.
{\em J. Differential
Equations} 234 (2007), 360-390.

\bibitem{CERW1} C. Cortazar, M. Elgueta, J. Rossi, N. Wolanski, 
\newblock How to approximate the heat equation with Neumann
boundary conditions by nonlocal diffusion problems.
{\em  Arch. Rat. Mech. Anal.} 187 (2008), 137-156.

\bibitem{CERW2} C. Cortazar, M. Elgueta, J. Rossi, N. Wolanski,
\newblock  Asymptotic behavior for nonlocal diffusion equations.
{\em J. Math. Pures Appl.} 86 (2006), 271-291.


\bibitem{CT} A. Cotsiolis, N. K. Tavoularis,
Best constants for Sobolev inequalities for higher order fractional derivatives. {\em J. Math. Anal. Appl.} 295 (2004) 225-236.

\bibitem{p2} L. Del Pezzo, A. Quaas, 
\newblock A Hopf's lemma and a strong minimum principle for the fractional p -Laplacian. {\em J. Differential Equations} 263  (2017),  no. 1, 765-778.

\bibitem{dibenedetto} 
E. DiBenedetto, {\em Partial Differential Equations (Second edition)}, Birkhäuser. 2010

\bibitem{DnPV} E.~Di Nezza, G.~Palatucci, E.~Valdinoci,
\newblock Hitchhiker's guide to the fractional Sobolev spaces.
{\em Bull. Sci. Math.}, 136 (2012), no. 5, 521-573.



\bibitem{DG}S. Dipierro, H. C. Grunau,
\newblock Boggio's formula for fractional polyharmonic {D}irichlet  problems. {\em Ann. Mat. Pura Appl. (4)} 196 (2017), no.4, 1327-1344.

\bibitem{DMPS}{S. Dipierro, L. Montoro, I. Peral and B. Sciunzi},
\newblock  Qualitative properties of positive solutions to nonlocal critical problems involving the Hardy-Leray potential. {\em Calc. Var. Partial Differential Equations},
55 (2016), no.4, Paper No. 99, 29.

\bibitem{DSV}{S. Dipierro, N. Soave, E. Valdinoci},
\newblock On fractional elliptic equations in Lipschitz sets and epigraphs: regularity, monotonicity and rigidity results. {\em Math. Ann.},  369  (2017),  no. 3-4, 1283-1326.

\bibitem{DRoV}  S.~Dipierro, X.~ Ros-Oton, E.~Valdinoci,
\newblock Nonlocal problems with Neumann boundary conditions.
{\em Rev. Mat. Iberoam.}, 33 (2017), no. 2, 377-416.

\bibitem{p3} G. Franzina, G. Palatucci, 
\newblock Fractional {$p$}-eigenvalues.
{\em Riv. Math. Univ. Parma (N.S.)} 5 (2014), no. 2, 373-386. 

\bibitem{G0} G. Grubb,
\newblock Fractional {L}aplacians on domains, a development of
 {H}\"ormander's theory of {$\mu$}-transmission pseudodifferential operators. {\em Adv. Math.}, 268 (2015), 478-528.

\bibitem{G} G. Grubb,
\newblock Local and nonlocal boundary conditions for  {$\mu$}-transmission and fractional order elliptic
pseudodifferential operators.  {\em Anal. PDE.}, 7 (2014), no. 7, 1649-1682.



\bibitem{G1} G. Grubb,
\newblock  Spectral results for mixed problems and fractional elliptic operators. {\em J. Math.
Anal. Appl.}, 421 (2015), no. 2, 1616-1634.

\bibitem{Guan} Q.Y. Guan,
\newblock Integration by Parts Formula for Regional Fractional
Laplacian. \emph{Commun. Math. Phys.}, 266  (2006),  289-329.

\bibitem{Q-YM}{Q. Y. Guan, Z. M.  Ma},
\newblock Reflected symmetric {$\alpha$}-stable processes and regional fractional {L}aplacian. {\em Probab. Theory Related Fields} 134 (2006), no. 4, 649-694.


\bibitem{LNT} C.S Lin, W.M. Ni, I. Takagi, 
 \newblock Large amplitude
stationary solutions to a chemotaxis systems. \emph{J. Differential Equations}, 72 (1988),  1-27.

\bibitem{MYZ} C. Miao, J. Yang, J. Zheng,
\newblock An improved maximal inequality for 2D fractional order Schr\"odinger operators. {\em Studia Math.} 230 (2015), no. 2, 121-165.


\bibitem{MPV}E. Montefusco, B. Pellacci, G. Verzini,
\newblock Fractional diffusion with Neumann boundary conditions: the
logistic equation. {\em Disc. Cont. Dyn. Syst. Ser. B} 18 (2013), 2175-2202.

\bibitem{Mou} C. Mou, Y. Yi,
\newblock Interior Regularity for Regional Fractional Laplacian. \emph{Commun. Math. Phys.},    340   (2015),  233-251.

\bibitem{RoS0} X. Ros-Oton, J. Serra,
\newblock The Dirichlet problem for the fractional Laplacian: regularity up to the boundary. {\em J. Math. Pures Appl},. 101 (2012), 275-302.


\bibitem{RoS1} X. Ros-Oton, J. Serra,
\newblock The Pohozaev identity for the fractional Laplacian. {\em Arch. Rat. Mech. Anal}. 213 (2014), 587-628.


\bibitem{RoS} X. Ros-Oton, J. Serra,
\newblock Local integration by parts and Pohozaev identities
for higher order fractional Laplacians. {\em Discrete Contin. Dyn. Syst.}, 35 (2015), no. 5, 2131-2150.

\bibitem{SV} P. Stinga, B. Volzone,
\newblock Fractional semilinear {N}eumann problems arising from a fractional {K}eller-{S}egel model. {\em Calc. Var. Partial Differential Equations} 54 (2015), no. 1, 1009-1042.



\bibitem{re} K. Rektorys,
\newblock Variational methods in mathematics scinece and engineering. {\em D. Reidel Publishing Company}. Boston. USA
\bibitem{V} G.C. Verchota,
\newblock  The biharmonic Neumann problem in Lipschitz domains. {\em Acta Math.}, 195  (2005), no. 2, 217--279.

\bibitem{Will} Michel Willem, 
\newblock Functional Analysis. Birkh\"auser Basel (2013)

\bibitem{Y} R. Yang,
\newblock On higher order extensions for the fractional Laplacian. {\em Preprint},  	arXiv:1302.4413.

\end{thebibliography}
\end{document}